\documentclass[12pt]{amsart}
\usepackage{graphicx, color, enumerate}
\usepackage{amsfonts,amsthm,amsmath,tabularx,epsf}
\usepackage{listings}
\usepackage{booktabs}
\usepackage{mathtools}
\usepackage{url,hyperref}
\usepackage{pgf,tikz,pgfplots}
\usepackage{algorithm}
\usepackage{algpseudocode}  
\usepackage{mathrsfs}
\usepackage[utf8]{inputenc}
\usepackage[T1]{fontenc}
\usepackage[english]{babel}
\usepackage{float}
\newtheorem*{remark}{Remark}
\theoremstyle{plain}% Theorem-like structures provided by amsthm.sty
\newtheorem{theorem}{Theorem}[section]
\newtheorem{lemma}[theorem]{Lemma}

\newtheorem{corollary}[theorem]{Corollary}
\newtheorem{proposition}{Proposition}[section]
\newtheorem{conjecture}[theorem]{Conjecture}
\theoremstyle{definition}
\newtheorem{definition}[theorem]{Definition}
\newtheorem{example}[theorem]{Example}

\theoremstyle{remark}

\usepackage[top=3cm,left=3cm,right=3cm,bottom=3cm]{geometry}
\usetikzlibrary{positioning,arrows,calc}
\tikzset{
	modal/.style={>=stealth,shorten >=1pt,shorten <=1pt,auto,node distance=1.5cm,
		semithick},
	world/.style={circle,draw,minimum size=0.5cm,fill=gray!15},
	point/.style={circle,draw,inner sep=0.5mm,fill=black},
	reflexive above/.style={->,loop,looseness=7,in=120,out=60},
	reflexive below/.style={->,loop,looseness=7,in=240,out=300},
	reflexive left/.style={->,loop,looseness=7,in=150,out=210},
	reflexive right/.style={->,loop,looseness=7,in=30,out=330}
}
\usetikzlibrary{decorations.markings}
\definecolor{cadmiumgreen}{rgb}{0.0, 0.42, 0.24}

\newcommand{\gpp}{gpp}
\newcommand{\gpps}{gpps}

\pgfplotsset{compat=1.18}

\begin{document}
	\pagestyle{myheadings}
	
	\title[Most A$_\alpha$-eigenvalues of a tree are small]{Most A$_\alpha$-eigenvalues of a tree are small}
	%%%%%%%%%%%%%%%%%%%%%%%%%%%%%%%%%%%%%%%%%%%%%%%%%%%%%%% nome

	\author[E. P. Bartur\'{e}n]{Elvia P. Bartur\'{e}n}
	\address{Universidad Nacional Mayor de San Marcos
		Escuela Profesional de Matem\'{a}tica,	Av. Venezuela s/n, cuadra 34, Lima 1--Per\'{u}; Ciudad Universitaria, UNMSM}
	\email{\tt eperezb@unmsm.edu.pe}

	\author[E. R. Oliveira]{Elismar R. Oliveira}
	\address{Instituto de Matem\'atica e Estat\'{\i}stica, UFRGS, Porto Alegre, Brazil}
	\email{\tt elismar.oliveira@ufrgs.br}

	\begin{abstract}
		We study the distribution of $A_\alpha$-eigenvalues of a tree. We extend the works \cite{Jacobs2021} and \cite{SIN2020} by proving that the number of $A_\alpha$-eigenvalues, for $0\leq \alpha \leq \frac{1}{2}$, less than or equal to the average degree $d_\alpha= \alpha (2- 2/n)$ is, at least, $\lceil \frac{n}{2} \rceil$ for a tree with $n$ vertices. Several counterexamples for $\frac{1}{2} <  \alpha \leq 1$ are exhibited. We also derive the same result for other well-known families such as the Deformed Laplacian and the matrix $B_\beta$.
	\end{abstract}
	
	\maketitle
	%%%%%%%%%%%%%%%%%%%%%%%%%%%%%%%%%%%%%%%%%%%%%%%%%%%%%%%%%%%%%%%%%%%%%%%%%%%%%%%%%%%%%%%%%%%%%%%%%%%%%%%%%%%%%%%%%%%%%%%%%%%%%%%%%%%%%%%%%%%%%

	\section{Introduction}%\label{sec:intro}
	The distribution of the Laplacian eigenvalues of a graph in a given interval (and its variations) is a prominent problem in spectral graph theory (see, for example \cite{mohar1992laplace}).  
	
	We aim to extend the results in \cite{Jacobs2021} and \cite{SIN2020} to the eigenvalues of the $A_\alpha$ matrix of a tree. The underlying Laplacian conjecture was formulated in \cite{jacobs2018conjecture}. It is worth mentioning the recent work of \cite{junior2025}, who investigated the spectral distribution of $A_\alpha(G)$-eigenvalues, over subintervals of the real line, establishing lower and upper bounds on the number of such eigenvalues in terms of structural parameters of a graph $G$, including the number of pendant and quasi-pendant vertices, the domination, the matching, and the edge covering numbers.
	
	We also refer to \cite{grone1994laplacian2,guo2008limit,guo2011distribution,jacobs2018conjecture,hedetniemi2016domination,merris1991number,Jacobs2021,SIN2020} and references therein for related problems and applications.

	Our main result is 
	\begin{theorem}
		\label{thr:main}
		Let $0\leq \alpha\leq 1/2$. For any tree $T$ of order $n \geq 1$,
		$$m_T(\alpha)\left(-\infty, d_\alpha \right] \geq \left\lceil\frac{n}{2} \right\rceil.$$
	\end{theorem}
	Since there are $n$ $A_\alpha$-eigenvalues bounded by $\Delta$, this is equivalent
	to
	$$
	\sigma_\alpha(T)=m_T(\alpha) \left(d_\alpha , \Delta\right]  \leq ~\left \lfloor\frac{n}{2} \right \rfloor.
	$$
	
	\bigskip
	The paper is organized as follows. In Section~\ref{sec:not} we introduce the notation used in the paper.  In Section~\ref{sec:pre} we establish the technical results and techniques need for the remaining of the work.   Section~\ref{sec:stra} explains the strategy and algorithmic tools used to reduce the general problem to a finite set of prototype trees. In Section~\ref{sec:conjecture} we present a conjecture on the $A_\alpha$-eigenvalues of a tree as a first approximation of the result in \cite{Jacobs2021}.  We also study the recurrences appearing from the eigenvalue location via diagonalization in order to determine counterexamples and finer properties of the $A_\alpha$-eigenvalues. In Section~\ref{sec:narrow}  we narrow the validity interval for the Conjecture~\ref{thm:main-conjecture} which leads to the right formulation of Theorem~\ref{thr:main}.  In Section~\ref{sec:proper}
	we present the proper transformations, which are performed on a tree $T$, in order to reduce it to a prototype tree, and prove the range of their validity. In Section~\ref{sec:prototype-trees} we introduce the prototype trees and prove that they satisfy the property in Theorem~\ref{thr:main}.  In Section~\ref{sec:reduction} we discuss the adaptation of the combinatorial reasoning from \cite{Jacobs2021} to reduce the number of starlike vertices until zero or one, which is the core argument in our strategy. Finally, in Section~\ref{sec:further} we discuss generalizations for other known families of matrices which are reparametrized as $A_\alpha$, so we can transfer our result to these families. 
	
	\section{Notation} \label{sec:not}
	
	We will adopt, to some degree, the notation from \cite{Jacobs2021}. However, for the benefit of the reader, we will recall some of it in this section.
	
	\begin{definition}
		Let $u$ be a vertex of a tree $T$ with degree $\geq 3$. Suppose that $P_q = u\,
		u_1 \ldots u_q~~ (q \geq 1)$ is a path in $T$ whose internal vertices
		$u_1,\ldots, u_{q-1}$ all have degree 2 in $T$, and where $u_q$ is a leaf.
		Then we say that $P_q$ is a pendant path of length $q$ attached at $u$.
	\end{definition}
	
	The following lemma is known from \cite{MoharLapCoeffs}.
	\begin{lemma}
		\label{lemma-starlike} Any tree that is not a path has at least one vertex
		$u$, with $\deg(u) \geq 3$, having (at least) two pendant paths.
	\end{lemma}
	
	\begin{definition}
		Let $T \neq P_n$ be a tree with $n$ vertices, and let $u$ be a vertex of
		degree at least $\ell + 1$ of $T$ having $\ell \geq 1$ pendant paths attached at
		$u$. We denote the \emph{sum} of pendant paths attached at $u$ by
		$P(u)=P_{q_1}\oplus\cdots \oplus P_{q_\ell}$. The number of edges in each path is denoted by $\sharp
		P_{q}=q$.
	\end{definition}
	
	Pendant paths of length 2 are key to our strategy.
	\begin{definition}
		A subgraph obtained by a vertex $u$ attached to $r\geq 1$ paths of length 2,
		is called a {\em sun with $r$ rays} and denoted by $S_r$.
		Hence, if a vertex $u$ of degree at least $r + 1$ has $r\geq 1$ pendant
		paths of length 2, say $P(u) = r P_2$, we will write
		$P(u)=S_r$.
	\end{definition}
	
	\begin{definition}
		Given a tree $T$, we call $P(u)=P_q \ast S_r$ a \emph{generalized pendant path at $u$}, abbreviated by {\em \gpp}, where $u$ is a vertex of degree $\geq3$ connected to $v$ by a path of length  $q \geq 0$, and $P(v) = r P_2$.  We note that $P(u)=P_q \ast S_r$ is a branch at $u$ (a maximal subtree).
	\end{definition}
	
	\begin{definition}
		We say a vertex $u$ is a \emph{starlike vertex} if it has degree $\geq 3$ and
		has at least two generalized pendant paths attached to it.
	\end{definition}
	This definition depends on the particular $(P_q, S_r)$
	representation, as well as the graph itself.
	
	\begin{remark}
		We call this a $(P_q, S_r)$ {\em representation of $T$}. We adopt the following convention:
		\[P_1  =  P_1 \ast S_0, \;  P_2 = P_0 \ast S_1 \text{ and }
		P_q  =  P_{q-2} \ast S_1  \text{ for } q \geq 3.\]
		This convention provides a particular $(P_q, S_r)$ representation
		of paths in a given tree in which $r$ is always equal to 0 or 1.
	\end{remark}
	
	\begin{definition}
		Let $T$ be a tree with $n$ vertices, $T \neq P_n$ and let $u$ be a starlike
		vertex of $T$ having $\ell \geq 1$ generalized pendant paths attached at $u$.
		According to our $(P_q,S_r)$ representation, we have  $P(u)=P_{q_1} \ast
		S_{r_1}\oplus\cdots \oplus P_{q_\ell} \ast S_{r_\ell}$. Let us call the
		\emph{weight} of $u$ and denote by $w(u)$, the sum of the number of vertices
		of the generalized pendant paths, that is
		$$w(u) =\sum_{i=1}^{\ell} \sharp (P_{q_i}\ast S_{r_i})
		=\sum_{i=1}^{\ell} (q_i+2r_i).$$
	\end{definition}

	\section{Preliminaries}\label{sec:pre}
	
	We consider the $A_{\alpha}=(1-\alpha) A+ \alpha D$ matrix, introduced by Nikiforov in \cite{nikiforov2017Aalpha} whose purpose was to merge the properties of the adjacency and Laplacian matrices of a graph (see \cite{wang2020hoffman, WangLiuBel2020} for additional spectral properties of this matrix). We recall that $A_0=A$, the adjacency matrix, $2A_{\frac{1}{2}} = Q$, the signless Laplacian, and for the last, $A_1=D$, the degree matrix. For trees, $Q$ and $L$ have the same spectrum; therefore, we recover the results from \cite{Jacobs2021} by choosing $\alpha=\frac{1}{2}$. Moreover, the interval $0\leq \alpha \leq 1/2$ will be of main interest since it connects the adjacency to the Laplacian properties. The interval $\alpha>1/2$ is expected to exhibit some anomalous behavior. 
	
	Two average degrees can be considered, by analogy with the
	Laplacian case. We recall that, for the Laplacian matrix, the average degree is  $d_n=2-\frac{2}{n}$.
	\begin{definition} We denote the average of the $A_\alpha$-eigenvalues by $d_\alpha$.
		\[
		d_{\alpha}
		:= \frac{1}{n}\operatorname{tr}(A_{\alpha}) = 
		\frac{\alpha}{n}
		\sum_{i=1}^{n} \operatorname{deg}_T(v_i),
		\]
		where $n=|T|$, the $v_i$ are the vertices of $T$, and $\operatorname{deg}_T(v_i)$ is the degree of $v_i$.
	\end{definition}
	The next lemma provides a straightforward relation between these two numbers.
	\begin{lemma}
		The following identity holds: $d_\alpha= \alpha d_n$.
	\end{lemma}
	\begin{proof}
		We can see that
		\begin{align*}
			d_{\alpha}
			&=\frac{1}{n}\operatorname{tr}(A_{\alpha})=\frac{1}{n}\operatorname{tr}\bigl[(1-\alpha)A+\alpha D\bigr]\\
			&=\frac{1}{n}\left[(1-\alpha)\operatorname{tr}(A)
			+\alpha\operatorname{tr}(D)\right]=\alpha\frac{1}{n}\operatorname{tr}(D)=\alpha d_n,
		\end{align*}
		where
		\[
		d_n=\frac{2m}{n}
		=\frac{2(n-1)}{n}
		=2-\frac{2}{n},
		\]
		since, for a tree, the number of edges is $m=n-1$.
		Thus, $d_{\alpha}=\alpha d_n$.
	\end{proof}
	
	From \cite[Proposition 9]{nikiforov2017Aalpha}, we know that the greatest $A_\alpha$-eigenvalue is bounded above by the maximum vertex degree $\Delta$ and the smallest  $A_\alpha$-eigenvalue is bounded above by $\alpha\delta$, where $\delta$ is  the minimum vertex degree, for every $\alpha\in[0,1]$.

	We now establish the main property in this work.
	\begin{definition}
		Let $T$ be a tree and $\lambda\in\operatorname{Spec}(A_{\alpha}(T))$. We say that $\lambda$
		is \emph{small} if
		$\lambda\leq d_{\alpha}$. Otherwise, we say that an $A_\alpha$-eigenvalue is
		\emph{large} if it is greater than $d_\alpha$.\\
		We also define
		\[
		m_T(\alpha)
		\bigl(-\infty,d_\alpha\bigr]
		=
		\sharp
		\left\{
		\lambda\in\operatorname{Spec}(A_\alpha(T))
		: \lambda\leq d_\alpha
		\right\},
		\]
		and
		\[
		\sigma_\alpha(T) 
		=
		\sharp
		\left\{
		\lambda\in\operatorname{Spec}(A_\alpha(T))
		:d_\alpha< \lambda \leq \Delta
		\right\},
		\]
		Following the notation of \cite{das2016distribution}, let
		\[
		\rho_{A_\alpha}(T)=\lambda_1(\alpha)
		\geq\lambda_2(\alpha)\geq\cdots\geq\lambda_n(\alpha).
		\]
		Then, $$
		m_T(\alpha)\bigl(-\infty,d_{\alpha}\bigr]
		+
		\sigma_\alpha(T)
		=n.$$
	\end{definition}
	
	\begin{remark}
		We may have
		$
		d_{\alpha}\in\operatorname{Spec}(A_{\alpha})$. For instance, if $\alpha=0$ then $d_\alpha=0 \in \operatorname{Spec}(A_{0})$ for $T=K_{1,n}, n \geq 2$. This provides an example (see Section~\ref{sec:narrow},  Example~\ref{ex:counter-example-individual} and   Proposition~\ref{prop:counterexamples} for nontrivial cases).
		For the Laplacian matrix, when $n\geq3$, one has
		$d_n \not\in\operatorname{Spec}(L)$, since $d_n=2-2/n$ is rational but not an integer and every Laplacian eigenvalue is an algebraic integer.
	\end{remark}

	\section{Strategy and algorithmic tools}\label{sec:stra}
	
	Our maintool is the notion of proper transformation.
	\begin{definition}
		A \emph{proper transformation} is defined  as a local operation on a tree $T$
		which produces a new tree $T^\prime$, with the same number of vertices, that
		does not decrease the number of $A_{\alpha}$-eigenvalues above the average degree, that is
		$$\sigma_\alpha(T) \leq \sigma_\alpha(T^\prime).$$ 
	\end{definition}
	We mainly use proper transformations for $n \geq7$; however, trees of smaller size will be  checked by direct computation.
	
	Our strategy to prove that, for any tree $T$ of order $n$, $\sigma_\alpha(T) \leq
	\lfloor\frac{n}{2}\rfloor$ is to apply successive transformations to $T$ 
	producing a prototype tree $T^\prime$. From the fact that we use proper
	transformations, and that $\sigma_\alpha(T^\prime) \leq \lfloor\frac{n}{2}\rfloor$,
	it follows that $\sigma_\alpha(T) \leq \sigma_\alpha(T^\prime)\leq \lfloor\frac{n}{2}\rfloor$, proving
	Theorem~\ref{thr:main}.

	Everything starts with the initialization procedure \texttt{InitiateRepresentation$(T)$} that puts the tree $T$ into a $(P_q, S_r)$ representation. It is formally described by the pseudocode of Figure~\ref{fig:init-rep}.
	\begin{figure}[ht!]
		\centering
		{\small
			{\tt
				\begin{tabbing}
					aaa\=aaa\=aaa\=aaa\=aaa\=aaa\=aaa\=aaa\= \kill
					
					\> \texttt{InitiateRepresentation}($T$)\\
					
					\> {\bf input}: a tree $T$ with $n$ vertices.\\
					
					\> {\bf output}: a $(P_q,S_r)$ representation of $T$.\\
					
					\> \\
					
					\> \> Identify all $j$ pendant paths
					$P_{q_1},\ldots,P_{q_j}$ of $T$.\\
					
					\> \> {\bf for} $i$ {\bf from} $1$ {\bf to} $j$ {\bf do}.\\
					
					\> \> \> {\bf if} $q_i=1$ {\bf then}
					replace $P_{q_i}$ with \gpp\; $P_1\ast S_0$;\\
					
					\> \> \> {\bf if} $q_i=2$ {\bf then}
					replace $P_{q_i}$ with \gpp\; $P_0\ast S_1$;\\
					
					\> \> \> {\bf if} $q_i\geq3$ {\bf then}
					replace $P_{q_i}$ with \gpp\; $P_{q_i-2}\ast S_1$;\\
					
					\> \> {\bf end do}\\
					
					\> {\bf return} $T$.
					
				\end{tabbing}
			}
		}
		\caption{Initiating the representation of a tree $T$}
		\label{fig:init-rep}
	\end{figure}

	We then describe a high-level algorithm \texttt{Transform} to perform the transformation sequence, shown in Figure~\ref{fig:transform}.
	
	\begin{figure}[ht!]
		\centering
		{\small {\tt
				\begin{tabbing}
					aaa\=aaa\=aaa\=aaa\=aaa\=aaa\=aaa\=aaa\= \kill
					
					\> Algorithm \texttt{Transform}$(T)$ \\
					
					\> Input: a tree $T$ with $n\geq 7$ vertices.\\
					
					\> Output: a tree $T^\prime$ with $n$ vertices and
					$\sigma_\alpha(T) \leq \sigma_\alpha(T^\prime)$ \\
					
					\> \\
					
					\> \> initialize: $T$:=InitiateRepresentation$(T)$.\\
					
					\> \> order starlike vertices of $T$ as
					$u_1,\ldots,u_k$ by weight,
					$w(u_1)\leq\cdots\leq w(u_k)$. \\
					
					\> \> {\bf while} $k\geq 2$ {\bf do}\\
					
					\> \> loop invariant:
					$w(u_1)\leq 2\left\lfloor\frac{n}{4}\right\rfloor$  \\
					
					\> \> \> $T$ := ReduceStarVertex$(T,u_1)$. \\
					
					\> \> \> order starlike vertices of the updated $T$ as
					$u_1,\ldots,u_k$ by weight,
					$w(u_1)\leq\cdots\leq w(u_k)$. \\
					
					\> \> {\bf end loop} \\
					
					\> \> {\bf return} $T$ \\
					
				\end{tabbing}
		}}
		\caption{Transformation of a tree $T$}
		\label{fig:transform}
	\end{figure}
	The next step inside the algorithm \texttt{Transform} is the identification and ordering of all $k$ starlike vertices
	of $T$. Recall that the weight of a starlike vertex $u$ is the total number of vertices hanging at $u$. The main parameter of our transformation algorithm is the number $k$ of starlike vertices and their weights. The next step is the procedure \texttt{ReduceStarVertex} $(T,u_1)$. It takes the tree $T$ and its starlike vertex of minimum weight $u_1$ as arguments, and properly transforms the generalized pendant paths at $u_1$ into a single generalized pendant path. 
	\begin{center}
		\begin{figure}[ht!]
			{\small {\tt
					\begin{tabbing}
						aaa\=aaa\=aaa\=aaa\=aaa\=aaa\=aaa\=aaa\= \kill
						\> \texttt{ReduceStarVertex}($T,u$)\\
						\> {\bf input}: a tree $T$ with $n$ vertices \\
						\> \> \> ~ a starlike vertex $u$ with $P(u)= P_{q_1}\ast S_{r_1} \oplus \cdots \oplus P_{q_\ell}\ast S_{r_\ell}$. \\
						\> \> \> ~ precondition $w(u) \leq 2 \lfloor \frac{n}{4} \rfloor$   \\
						\> {\bf output}: a tree $T^\prime$ where the {\gpps} at $u$ are replaced by a single \gpp. \\
						\> \> Let $w=w(u)$. \\
						\> \> Compute $q'=\sum_i q_i\pmod 2$. \\
						\> \> Compute $r'=\frac{w-q'}{2}$.      \\
						\> \> Replace in $T$ all {\gpps} at $u$ with $P(u) = P_{q'}\ast S_{r'}$, forming $T'$.\\
						\> \> {\bf if}  $\deg_{T'} (u) = 2$ {\bf then} \\
						\> \> \>  find $v$, the nearest vertex to $u$ having $\deg_{T'}(v) > 2$ \\
						\> \> \>  temporarily remove $u$ \\
						\> \> \>  create {\gpp} $P(v)=P_{q^{\prime\prime}}\ast(P_{q'}\ast S_{r'})=P_{q^{\prime\prime}+q'}\ast S_{r'}$, where $d(u,v)=q^{\prime\prime}$ \\
						\>{\bf return} $T'$.
				\end{tabbing}}
				\caption{Procedure \texttt{ReduceStarVertex}.}\label{red-star}}
		\end{figure}
	\end{center}
	Moreover, we will see that this does not increase the number of starlike vertices, but increases the minimum weight over time. This guarantees termination. The goal is that, after it stops, we end with no starlike vertices or exactly one starlike vertex.
	
	We now review our main technical tool for proving the conjecture. It is the algorithm
	reproduced in Figure \ref{alg-lap}. For any tree $T$ with $n$ vertices, it
	produces a diagonal matrix $D$ that is congruent to the matrix $A_\alpha-xI_n$.
	
	This algorithm, presented first in \cite{fritscher}, is the Laplacian matrix
	version of the adjacency matrix algorithm \cite{JT2011} that has been useful
	in many applications of spectral graph theory (see, for example, the recent
	ordering by the spectral radius \cite{oliveira2018,belardo} of certain trees). For the general case, we refer to \cite{braga2013distribution,livro} and references therein.
	\begin{figure}[ht!]
		\centering
		{\small
			{\tt
				\begin{tabbing}
					aaa\=aaa\=aaa\=aaa\=aaa\=aaa\=aaa\=aaa\= \kill
					
					\> The vertices are ordered so that every child precedes its parent. \\
					\> Algorithm $\operatorname{Diagonalize}(A_\alpha(T),-x)$ \\
					
					\> \> initialize $d(v_i):=\alpha\deg_T(v_i)-x$ for all vertices $v_i$ \\
					
					\> \> {\bf for} $k=1$ to $n$ \\
					
					\> \> \> {\bf if} $v_k$ is a leaf {\bf then} continue \\
					
					\> \> \> {\bf else if} $d(v_i)\neq0$ for all children $v_i \in C_k$ {\bf then} \\
					
					\> \> \> \> $d(v_k):=d(v_k)-\displaystyle\sum_{v_i\in C_k}\frac{(1-\alpha)^2}{d(v_i)}$ \\
					
					\> \> \> {\bf else} \\
					
					\> \> \> \> select one child $v_j$ of $v_k$ for which
					$d(v_j)=0$ \\
					
					\> \> \> \> $d(v_k):=-\frac{(1-\alpha)^2}{2}$ \\
					
					\> \> \> \> $d(v_j):=2$ \\
					
					\> \> \> \> if $v_k$ has a parent $v_l$, remove the edge
					$v_kv_l$. \\
					
					\> \> {\bf end loop}
					
				\end{tabbing}
			}
		}
		\caption{Diagonalization of $A_\alpha-xI_n$}\label{alg-lap}
	\end{figure}
	
	\begin{lemma}\cite[Theorem 3]{JT2011}\label{lemma-eigen}
		The number of $A_\alpha$-eigenvalues of $T$ less than, equal to, or greater than \(x\) is exactly the number of negative, zero, or positive diagonal elements produced by
		\(\operatorname{Diagonalize}(A_\alpha(T),-x)\).
	\end{lemma}
	\begin{remark}
		For our purposes, we take $x=d_\alpha$. Moreover, as we will often work with different trees, we adopt the following convention. We denote by $\deg_T(v)$ the degree of a vertex $v$ in the tree $T$. Also, the output of $\operatorname{Diagonalize}(A_\alpha(T),-x)$ at $v$ will be denoted by $a_T(v)$. This allows us to track a fixed vertex when a part of a tree $T$ is changed to $T'$, by some transformation allowing us to compare $\operatorname{Diagonalize}(A_\alpha(T),-x)$ and $\operatorname{Diagonalize}(A_\alpha(T'),-x)$.
	\end{remark}

	\section{\texorpdfstring{A conjecture on the $A_\alpha$-eigenvalues of a tree}{A conjecture on the A-alpha-eigenvalues of a tree}} \label{sec:conjecture}
	We now seek a first approximation of the result in \cite{Jacobs2021} for the  $A_\alpha$ matrix of a tree.

	\begin{conjecture}\label{thm:main-conjecture}
		Let $T$ be a tree such that $|T|=n \geq 1$. Then,
		\[ m_T(\alpha)\bigl(-\infty,d_{\alpha}\bigr]
		\geq \left\lceil \frac{n}{2} \right\rceil,
		\qquad
		\forall\,\alpha\in[0,1].
		\] 
	\end{conjecture}
	
	To avoid division by zero, we first study the validity of Conjecture~\ref{thm:main-conjecture} for the extremes $\alpha=0$ and $\alpha=1$. 
	
	\textbf{Case $\alpha=0$.}
	For $\alpha=0$, we have $A_0=A$. Since the spectrum of a tree is
	symmetric with respect to the origin, if $\lambda$ is an eigenvalue,
	then $-\lambda$ is also an eigenvalue. Therefore, there are at least $\left\lceil\frac{n}{2}\right\rceil$ eigenvalues less than or equal to $d_0=0$. Consequently,
	\[
	m_T(\alpha)\bigl(-\infty,d_0\bigr]
	\geq \left\lceil\frac{n}{2}\right\rceil.
	\]
	\textbf{Case $\alpha=1$.}
	We claim that the conjecture is false for $\alpha=1$. Indeed, in this
	case, $A_1=D$, where $D$ denotes the degree matrix of $T$.
	
	For example, for $T=P_5$, 
	$A_1=\operatorname{Diag}(1,2,2,2,1)$, $\operatorname{Spec}(A_1)=\{1^2,2^3\}$ and  $d_1=2-\frac{2}{5}=\frac{8}{5}<2$.
	In this case 
	\[m_T(\alpha)\bigl(-\infty, d_1\bigr] = 
	m_T(\alpha)\bigl(-\infty,d_1\bigr]=2<\left\lceil\frac{5}{2}\right\rceil=3.
	\]
	This proves the claim.
	
	\begin{remark}
		Since the conjecture is always true for $\alpha=0$ and false in general for $\alpha=1$, we may assume from now on that $0<\alpha<1$, aiming to prove the conjecture in this interval.
		
	\end{remark}
	Following our strategy, we study the location of $d_\alpha$ for a particular substructure that plays a fundamental role, the generalized pendant paths.
	
	Recall the explicit expression for $A_\alpha$: 
	\[
	A_\alpha(T) =
	\begin{pmatrix}
		\alpha\deg_T(v_1) & (1-\alpha) a_{12} & \cdots & (1-\alpha) a_{1n}\\
		(1-\alpha) a_{21} & \alpha\deg_T(v_2) & \cdots & (1-\alpha) a_{2n}\\
		\vdots & \vdots & \ddots & \vdots\\
		(1-\alpha) a_{n1} & (1-\alpha) a_{n2} & \cdots & \alpha\deg_T(v_n)
	\end{pmatrix}.
	\]
	We aim to locate the value $d_\alpha$ in the spectrum of $A_\alpha$ using the \texttt{Diagonalize} algorithm to a $gpp$ given by $P_q\ast S_r$. The initialization step of the algorithm is illustrated in Figure~\ref{fig:diagAalfa}. 
	\begin{figure}[ht!]
		\centering    \includegraphics[width=0.5\linewidth]{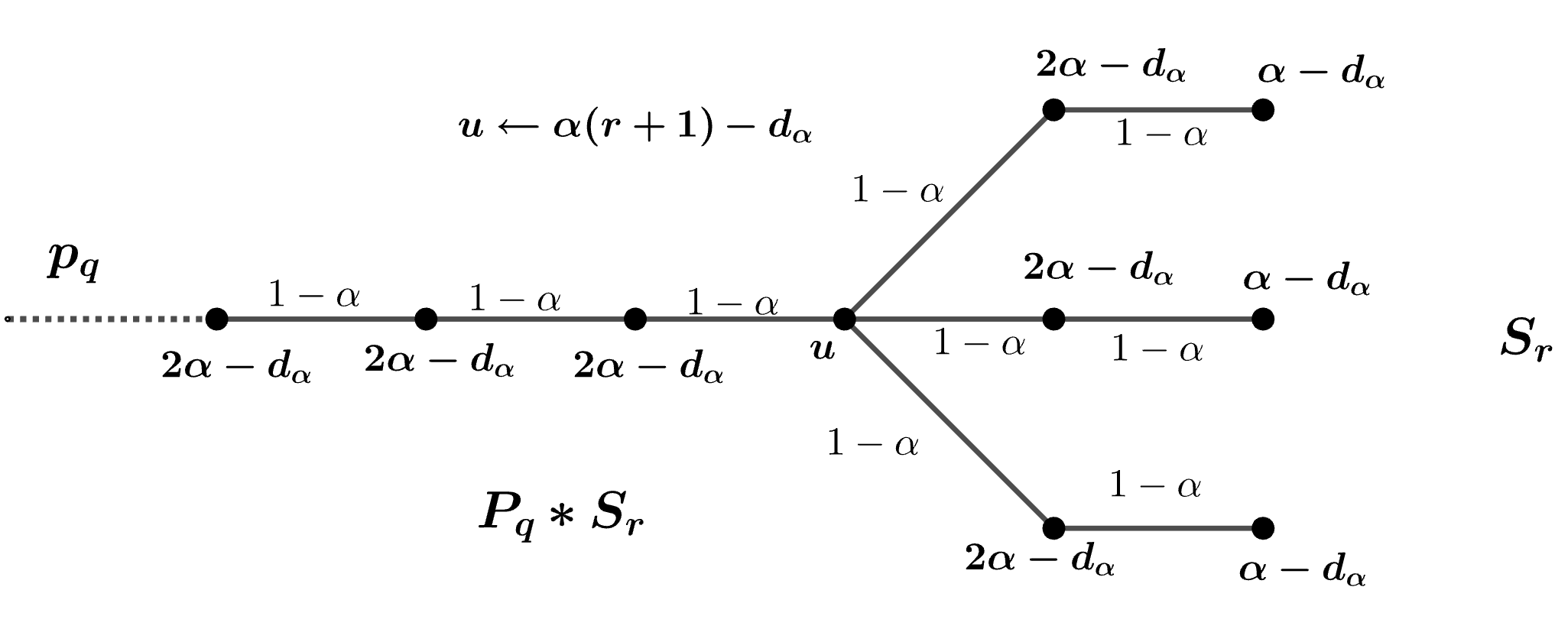}    \caption{$\operatorname{Diag}\bigl(A_\alpha(T),-d_\alpha \bigr)$}
		\label{fig:diagAalfa}
	\end{figure}
	We follow the process of the \texttt{Diagonalize} algorithm, which is shown in Figure~\ref{fig:diagproc}.
	\begin{figure}[ht!]
		\centering    \includegraphics[width=0.4\linewidth]{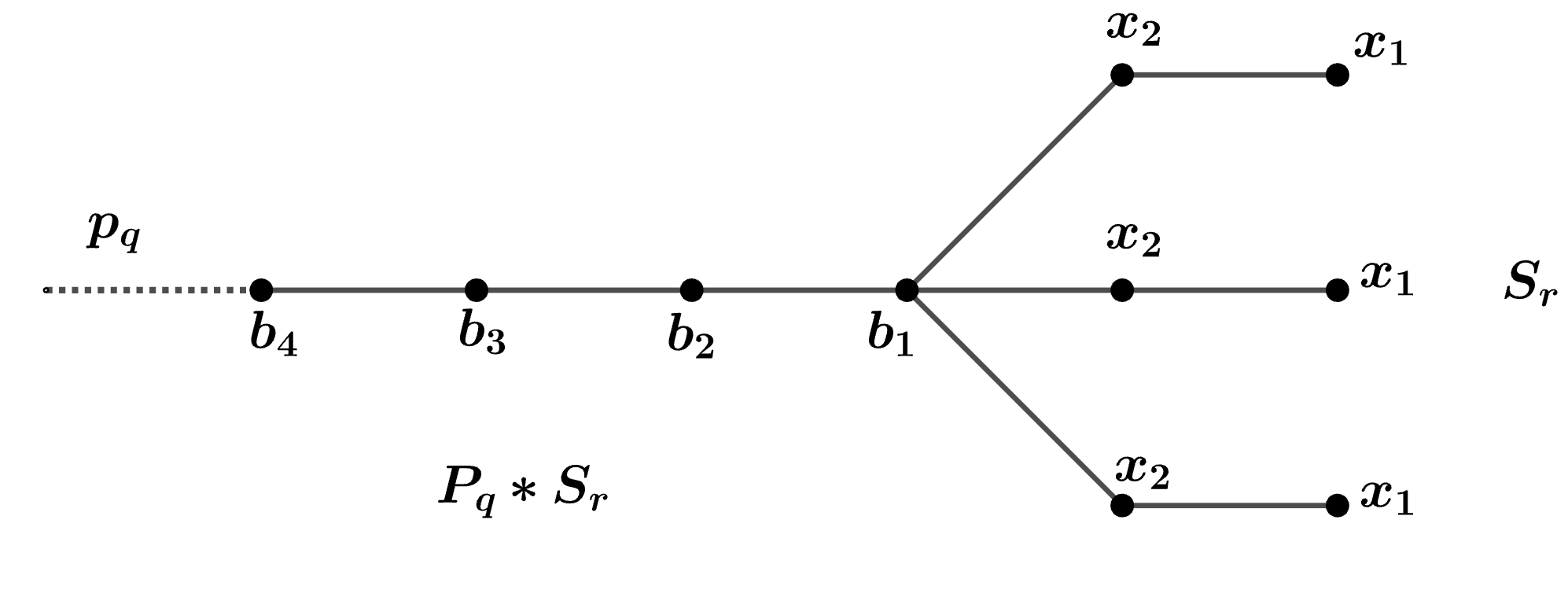}
		\caption{Step-by-step execution of the \texttt{diagonalize} algorithm.}
		\label{fig:diagproc}
	\end{figure}
	
	Applying the algorithm  for the $A_\alpha$ matrix, to locate $x = d_\alpha$, we obtain, at each extremal vertex of the pendant path $P_2$, the value $x_1(\alpha)$.
	
	For $\alpha > 0$ and $d_\alpha=\alpha(2-\frac{2}{n})$, we have
	\[
	x_1(\alpha) =\alpha-d_\alpha = \alpha\left(-1 + \frac{2}{n}\right) < 0
	\]
	and, if $x_1(\alpha)\neq 0$, then the next value processed is
	\[
	x_2(\alpha) = 2\alpha - d_\alpha - \frac{(1-\alpha)^2}{x_1(\alpha)} = \frac{2\alpha}{n} - \frac{(1-\alpha)^2}{x_1(\alpha)}.
	\]
	For completeness, consider the recurrence relation
	\begin{equation}
		\left\{
		\begin{aligned}
			x_1(\alpha) &= \alpha\big(-1 + \dfrac{2}{n}\big), \\[4pt]
			x_{j+1}(\alpha) &= \dfrac{2\alpha}{n} - \dfrac{(1-\alpha)^2}{x_j(\alpha)},0<\alpha<1,
		\end{aligned}
		\right.
		\label{eq:recorrencia}
	\end{equation}
	provided that $x_j(\alpha)\neq0$ for all $j$.

	Starting from these values, we continue by processing the vertices of the path $P_q$ shown in Fig.~\ref{fig:diagproc}, obtaining at $u$
	\[
	\begin{aligned}
		b_1(r) &= \alpha(r + 1) - d_\alpha - \frac{r(1-\alpha)^2}{x_2(\alpha)} = x_1(\alpha) + r\left(\alpha - \frac{(1-\alpha)^2}{x_2(\alpha)}\right).
	\end{aligned}
	\]
	and the rest of the values on the path are given by the recursion
	\begin{equation}
		\left\{
		\begin{aligned}
			b_1(r) & = x_1(\alpha) + r\left(\alpha - \frac{(1-\alpha)^2}{x_2(\alpha)}\right), \\
			b_{j+1}(r) & = \frac{2\alpha}{n} - \frac{(1-\alpha)^2}{b_j(r)},
		\end{aligned}
		\right.
		\label{eq:sistema_b}
	\end{equation}
	provided that $b_j(r) \neq 0, \forall j$. If $x_j(\alpha)=0$ or $b_j(r)=0$ then $d_\alpha$ could be a root of the characteristic polynomial of $A_\alpha$. See \cite{OliveTrevAppDiff} for a study of these recurrences and applications to spectral graph theory.
	\begin{remark}
		Since the sequence $b_j$ depends on $r$, we write
		$b_j=b_j(r)$ (we suppress the dependence on $\alpha$ in the notation) for $r\geq 0$. For $r=0$, we have $b_1(0)=x_1(\alpha)$, and, since $(b_j)$ satisfies the same recurrence as $(x_j)$, it follows that $b_j(0)=x_{j}(\alpha)$ for all $j\geq 1$. In particular, $b_1(1)=x_3(\alpha)$. 
	\end{remark}  
	The main properties of the sequences $(x_j(\alpha))_{j \in\mathbb{N}}$ and
	$(b_j(r))_{j \in\mathbb{N}}$ are summarized in the following lemma.
	\begin{lemma}\label{lemma-sequence} Let $T$ be a tree with $n\geq3$ vertices and $0<\alpha<1$.
		Consider the above defined sequences $x_j(\alpha)$ and $b_j(r)$.
		Then the following properties hold:
		\begin{enumerate}
			\item[(a)] If $n\geq3$, then $-\alpha<x_1(\alpha)\leq-\frac{\alpha}{3}$ and $x_2(\alpha)>\frac{(1-\alpha)^2}{\alpha}>0$.
			
			\item[(b)] If $n\geq3$, then the map $b_1(r): \mathbb{N}_0 \to [x_1(\alpha),\infty)$ given by  $r\mapsto b_1(r):= r\, m(\alpha) + x_1(\alpha)$ is linear and strictly
			increasing, where $m:=m(\alpha) =\alpha-\frac{(1-\alpha)^2}{x_2}>0$;    
			
			\item[(c)] Consider $n \geq 3$. Let $r_0>0$ be the unique positive root of $b_1(r)$ and let $r_1>r_0$ be the unique value such that $b_1(r)= \frac{(1-\alpha)^2}{2\alpha}n$. Then,
			\begin{enumerate}
				\item[(1)] The function $(\alpha,n) \mapsto r_0(\alpha,n)$ is strictly decreasing for $0<\alpha<1$ and $n\geq3$. Moreover,
				\[\lim_{\alpha \to 0^+} r_0(\alpha,n) =\frac{n-2}{2} \text{ and } \lim_{\alpha \to 1^-} r_0(\alpha,n) =1-\frac{2}{n}. \]
				In particular, 
				\[1-\frac{2}{n} < r_0(\alpha,n) < \frac{n-2}{2}.\]
				\item[(2)] If 
				\[\alpha:=\alpha_{n,k} =  \left( 1+\sqrt {{\frac { \left( 2n-4 \right)  \left( kn-n+2
							\right) }{{n}^{2} \left( n-2-2\,k \right) }}} \right) ^{-1}, \quad 1 \le k \le \left\lfloor \frac{n-3}{2} \right\rfloor.\]
				then $r_0=k \in \mathbb{N}$ and $b_2(k)$ is not defined;
				\item[(3)] If $n\geq7$ and $\alpha\leq1/2$, then $r_0>\frac{n}{4}$, otherwise, if $\alpha > 1/2$ then there exists $N:=N(\alpha) \in \mathbb{N}$ such that, $r_0 < \frac{n}{4}$ for $n \geq N(\alpha)$;
				\item[(4)] If  $0 \leq r<r_0$ then $b_1(r)<0$ and $b_2(r)>0$, and
				\item[(5)] If $r_0<r<r_1$ then $b_1(r)>0$ and $b_2(r)<0$.
			\end{enumerate}
		\end{enumerate} 
	\end{lemma}
	\begin{proof}
		\begin{enumerate}
			\item[a)]  We have $x_1(\alpha)=\alpha \left(-1+\frac{2}{n}\right)$, as $0<\frac{2}{n}\leq\frac{2}{3}$, so $-1<-1+\frac{2}{n}\leq-\frac{1}{3}$.
			Then, $-\alpha<x_1(\alpha)\leq-\frac{\alpha}{3}$ and $x_2(\alpha) = \frac{2\alpha}{n} - \frac{(1-\alpha)^2}{x_1}$ thus  $\quad x_2(\alpha) > 0$ because $\quad x_1(\alpha) < 0$. As  $ x_1>-\alpha$ one obtains $ -x_1<\alpha $ and consequently $\frac{-1}{x_1}>\frac{1}{\alpha}$ thus $x_2(\alpha) = \frac{2\alpha}{n} - \frac{(1-\alpha)^2}{x_1}>\frac{(1-\alpha)^2}{\alpha}$.
			Therefore, $x_2(\alpha)>\frac{(1-\alpha)^2}{\alpha}$.

			\item[b)] We know by item (a) that $x_2(\alpha)>0$, thus the number $b_1$, given by $Diagonalize(A_\alpha,-d_\alpha)$ is well-defined. Thus, we can compute
			\begin{align*}
				b_1(r)=a_T(u)=\alpha(1+r)-d_\alpha- \frac{r(1-\alpha)^2}{x_2(\alpha)}
				&= x_1(\alpha)+r\left( \alpha-\frac{(1-\alpha)^2}{x_2(\alpha)}\right)
			\end{align*}
			Denoting $m(\alpha)=\alpha-\frac{(1-\alpha)^2}{x_2(\alpha)}$, one obtains $b_1(r)=mr+x_1(\alpha)$.
			That is, the correspondence $r\mapsto b_1(r)$ is linear with respect to $r$, we note that from item (a)
			$x_2(\alpha)>\frac{(1-\alpha)^2}{\alpha}$ implies that $m>0$, therefore $b_1(r)$ is strictly increasing with range $\left[x_1(\alpha), \infty  \right)$, because $b_1(0)=x_1(\alpha)$.
			
			\item[c)] We now prove  each:
			\begin{enumerate}
				\item[1)]We recall from  item (b) that $b_1(r)$ has a unique positive root $r_0$ because it is increasing and $b_1(0)=x_1(\alpha)<0$. In this way $b_1(r_0)=0$ if and only if $ mr_0+x_1(\alpha)=0$ if and only if $  r_0(\alpha,n):=\frac{-x_1(\alpha)}{m(\alpha)}$ therefore
				\[r_0(\alpha,n)= {\frac { \left( n-2 \right)  \left(  \left( {n}^{2}+2\,n-4
						\right) {\alpha}^{2}-2\,\alpha\,{n}^{2}+{n}^{2} \right) }{2\,n \left( 
						\left( 2\,n-2 \right) {\alpha}^{2}-2\,\alpha\,n+n \right) }}.
				\]
				Thus, it is enough to observe that 
				\[\frac{\partial}{\partial \alpha} r_0(\alpha,n) = {\frac {\alpha\, \left(  \left( {n}^{3}-6\,{n}^{2}+12\,n-8 \right) 
						\alpha-{n}^{3}+6\,{n}^{2}-12\,n+8 \right) }{ \left(  \left( 2\,n-2
						\right) {\alpha}^{2}-2\,\alpha\,n+n \right) ^{2}}}=
				\] \[={\frac {\alpha\,
						\left( n-2 \right) ^{3} \left( \alpha-1 \right) }{ \left(  \left( 2\,
						n-2 \right) {\alpha}^{2}-2\,\alpha\,n+n \right) ^{2}}} <0
				\]
				and the limits $\lim_{\alpha \to 0^+} r_0(\alpha,n) =\frac{n-2}{2} \text{ and } \lim_{\alpha \to 1^-} r_0(\alpha,n) =1-\frac{2}{n}$ can be checked directly from the formula for $r_0$.
				\item[2)] We want to find when $r_0(\alpha,n)=k \in \mathbb{N}$. From the previous item, it is equivalent to $1-\frac{2}{n} < k < \frac{n-2}{2}$, that is, $1 \leq k \leq \lfloor\frac{n-3}{2}\rfloor$.
				
				Performing a change of variables $\alpha=\frac{1}{1+t}$ for $t \in (-1, +\infty)$ at $r_0$ we obtain 
				\[\tilde r_0(t,n):=r_0\left(\frac{1}{1+t},n\right)= {\frac { \left( {n}^{2}{t}^{2}+2\,n-4 \right)  \left( n-2
						\right) }{ 2 n \left( n{t}^{2}+n-2 \right)}}.
				\]
				Solving $\tilde r_0(t,n)=k$ with respect to $t$ yields $t=\sqrt {{\frac { 2\left( n-2 \right)  \left( kn-n+2 \right) }{{
								n}^{2} \left( n-2-2\,k \right) }}}.$
				Substituting that in $\alpha$ we have
				\[\alpha_{n,k}:= \left( 1+\sqrt {{\frac { \left( 2n-4 \right)  \left( kn-n+2
							\right) }{{n}^{2} \left( n-2-2\,k \right) }}} \right) ^{-1}.
				\]
				Thus, $r_0=k \in \mathbb{N}$ and $b_2(r)$ is not defined under the Diagonalize algorithm.
				\item[3)] We claim that $r_0>\frac{n}{4}$, for $n\geq 7$.
				\begin{figure}[ht!]
					\centering
					\includegraphics[width=0.3\linewidth]{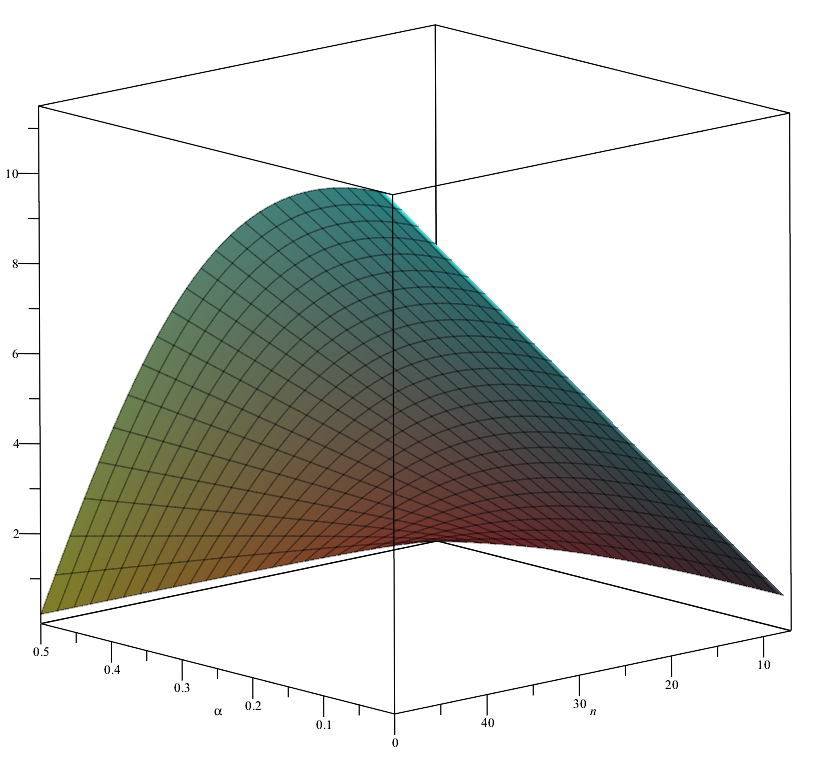}
					\caption{Plot of $r_0-\frac{n}{4}$}
					\label{fig:r0minusn4}
				\end{figure}
				
				Indeed, it is equivalent to show that $g(\alpha,n):=r_0(\alpha,n) - n/4>0$ (see Figure~\ref{fig:r0minusn4}). Note that
				$\frac{\partial}{\partial \alpha} g(\alpha,n) = \frac{\partial}{\partial \alpha} r_0(\alpha,n) <0.$ 
				On the other hand
				\[g(1/2, n)= \frac {{n}^{2}-8\,n+8}{4\,n \left( n-1 \right) }>0 \leftrightarrow n > 4+2\sqrt{2}\approx 6.82.\]
				Thus, for any $\alpha \in(0,1/2]$, $g(\alpha,n)>g(1/2,n)>0$. This proves the first claim. 
				
				For the second claim, we note that, for a fixed $n$,  $g(1/2, n)>0$ and $\frac{\partial}{\partial \alpha} r_0(\alpha,n) <0$ so that $r_0(\alpha,n) - n/4>0$ for a small interval $[1/2, 1/2+\varepsilon)$. However, we can find a value $\alpha(n)> 1/2$ such that $r_0(\alpha,n) - n/4=0$ and became negative after. Solving this equation with respect to $\alpha$ we have
				\[\alpha(n)={\frac { \left( {n}^{2}-4\,n -\sqrt {\left( n-2 \right)  \left( n-4 \right)  \left( {n}^{2}-4\,n+8 \right)}\right) n}{2{n}^{2}-16\,n+16}}.
				\]
				\begin{figure}[ht!]
					\centering
					\includegraphics[width=0.45\linewidth]{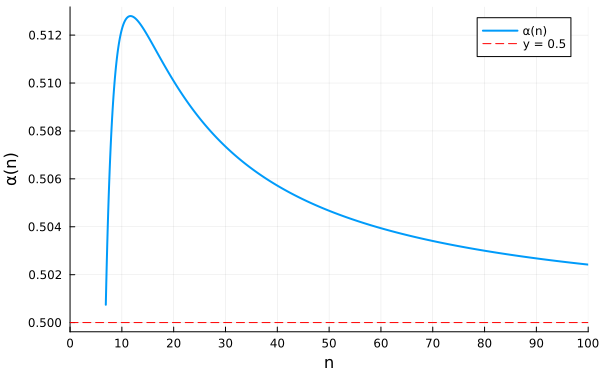}
					\caption{Plot of $\alpha(n)$ }
					\label{fig:r0minusn4not}
				\end{figure}
				One can check that $\lim_{n \to \infty} \alpha(n) = 1/2$, decreasingly (see Figure~\ref{fig:r0minusn4not}). Thus, for any $\alpha > 1/2$  there will be a number $N:=N(\alpha) \in \mathbb{N}$ such that, $r_0 < \frac{n}{4} $ for $n \geq N(\alpha)$.

				\item[4)]$0\leq r<r_0$ then $b_1(r)<0$ and $b_2(r)>0$
				
				This is an immediate consequence of the equations
				\[
				b_1(r) = mr + x_1(\alpha) \quad \text{and} \quad b_2(r) = \frac{2\alpha}{n} - \frac{(1-\alpha)^2}{b_1(r)}.
				\]
				Since $b_1(r)$ is strictly increasing, it is negative on the left of its root $r_0$. Thus, for $r < r_0$, we have $-b_1(r) > 0$, which implies $b_2(r) > 0$.
				
				\item[5)] If $r_0 < r < r_1$, then $b_1(r) > 0$ and $b_2(r) < 0$. To prove this, we note that the first inequality is trivial because $r$ lies on the right of the root $r_0$ of $b_1(r)$, implying that $b_1(r) > 0$. On the other hand, observe that
				\[b_2(r) < 0 \iff
				\frac{2\alpha}{n} - \frac{(1-\alpha)^2}{b_1(r)} < 0 \iff b_1(r) < \frac{(1-\alpha)^2 n}{2\alpha}.\]
				Therefore, if we define $r_1$ as the unique value such that $b_1(r_1) = \frac{(1-\alpha)^2 n}{2\alpha}$, this value is well-defined because $b_1(r)$ is strictly increasing and $r_0 < r_1$.
				
				Consequently, for $r_0 < r < r_1$, it follows that $b_2(r) < 0$. 
			\end{enumerate}
		\end{enumerate}
		This completes the proof.
	\end{proof} 
	
	We now resume the proof of Theorem~\ref{thr:main} by looking at small values of $n$, that is, whether the Conjecture~\ref{thm:main-conjecture} holds for $n \leq 6$. The next result is a direct consequence of Lemma~\ref{lemma-sequence} and the Diagonalization algorithm.
	We note that even for small values, the conjecture could be false if $\alpha>1/2$, see Table~\ref{table_P_5}.
	{\tiny \begin{table}[htbp]
			\centering
			\caption{Data table derived from the spectrum of $A_\alpha(P_5)$. The eigenvalues less than or equal to $d_\alpha$ are in bold.}
			\vspace{0.5em}
			\begin{tabular}{ccccccc}
				\toprule
				$\mathbf{\alpha}$ & $\text{d}_\alpha$ & $\mathbf{\lambda_5}$ & $\mathbf{\lambda_4}$ & $\mathbf{\lambda_3}$ & $\mathbf{\lambda_2}$ & $\mathbf{\lambda_1}$ \\
				\midrule
				0.0100 & 0.0160 & \textbf{-1.6964} & \textbf{-0.9750} & \textbf{0.0133} & 1.0050 & 1.7331 \\
				0.1000 & 0.1600 & \textbf{-1.3763} & \textbf{-0.7514} & \textbf{0.1334} & 1.0514 & 1.7430 \\
				0.2500 & 0.4000 & \textbf{-0.8475} & \textbf{-0.3853} & \textbf{0.3347} & 1.1353 & 1.7628 \\
				0.5000 & 0.8000 & \textbf{0.0000}  & \textbf{0.1910}  & \textbf{0.6910} & 1.3090 & 1.8090 \\
				0.7500 & 1.2000 & \textbf{0.6596}  & \textbf{0.6743}  & 1.2081 & 1.5757 & 1.8822 \\
				0.9000 & 1.4400 & \textbf{0.8888}  & \textbf{0.8890}  & 1.6650 & 1.8110 & 1.9463 \\
				0.9900 & 1.5840 & \textbf{0.9899}  & \textbf{0.9899}  & 1.9659 & 1.9801 & 1.9942 \\
				\bottomrule
			\end{tabular} \label{table_P_5}
	\end{table}}
	
	\begin{corollary}\label{cor:conjecture_small_n}
		Let $T$ be a tree with $|T|=n\leq6$ and $0<\alpha\leq\frac12$. Then the Conjecture~\ref{thm:main-conjecture} holds.
	\end{corollary}
	\begin{proof}
		We must examine each case individually.
		\begin{enumerate}
			\item If $n=1$, then there is only one non-isomorphic tree, namely
			$T=K_1=P_1$.  Thus, $0$ is the only eigenvalue of $A_\alpha(T)$. Therefore,
			\[
			m_T(\alpha)(-\infty,d_\alpha]
			= \left\lceil\frac{n}{2}\right\rceil=1,
			\]
			and the conjecture holds for $n=1$. 
			
			\item If $n=2$, then there is only one non-isomorphic tree on two
			vertices, namely $T=K_2=P_2$. Applying
			$\operatorname{Diagonalize}$, we obtain
			$x_1(\alpha)=0$ and we are not able to use Lemma~\ref{lemma-sequence}~(a). However, the algorithm produces $2$ and $-\frac{ (1-\alpha)^2}{2}$. Therefore, the conjecture holds, since $\lceil n/2\rceil=1$.
			
			\item If $n=3$, then there is only one non-isomorphic tree formed by
			three vertices, namely $T=P_3$.  We choose the internal vertex as the root, then we apply Diagonalize obtaining the value $x_1(\alpha)<0$ at each leaf, with two negative entries (see Equation~\ref{eq:recorrencia} and Lemma~\ref{lemma-sequence} (a)). The conjecture holds, since $\lceil n/2\rceil=2$.
			
			\item If $n=4$, then we have only two non-isomorphic trees formed by $T=P_4$ and $T=K_{1,3}$.  For $T=P_4$, we choose any internal vertex as the root, then we apply Diagonalize obtaining the value $x_1(\alpha)<0$ at each leaf, with two negative entries (see Equation~\ref{eq:recorrencia} and Lemma~\ref{lemma-sequence} (a)).  Analogously, if $T=K_{1,3}$ we choose as root the internal vertex and apply Diagonalize obtaining the value $x_1(\alpha)<0$ at each leaf, with three negative entries (see Equation~\ref{eq:recorrencia} and Lemma~\ref{lemma-sequence} (a)). In both cases, the conjecture holds, since $\lceil n/2\rceil=2$.
			
			\item If $n=5$, then there are three non-isomorphic trees, namely
			$T=P_5$, $T=K_{1,4}$, and the starlike tree $T=[1,1,2]$.
			
			For $T=P_5$, we choose the internal vertex $u$ as the root and
			apply $\operatorname{Diagonalize}$. At each leaf, we obtain
			$x_2(\alpha)>0$, obtaining two positive entries
			(see Equation~\ref{eq:recorrencia} and
			Lemma~\ref{lemma-sequence}~(a)). We then process the root vertex,
			obtaining
			\[
			a_T(u)=
			\frac{\alpha(32\alpha^2-40\alpha+20)}
			{31\alpha^2-50\alpha+25}>0.
			\]
			Similarly, for the starlike tree $T=[1,1,2]$, we choose the degree three
			vertex as the root and apply $\operatorname{Diagonalize}$. At each
			leaf, we obtain $x_1(\alpha)<0$, obtaining three negative entries
			(see Equation~\ref{eq:recorrencia} and
			Lemma~\ref{lemma-sequence}~(a)). Similarly, for $T=K_{1,4}$, we choose an arbitrary internal vertex as the root and
			apply $\operatorname{Diagonalize}$. At each leaf, we obtain
			$x_1(\alpha)<0$, obtaining four negative entries (see Equation~\ref{eq:recorrencia} and
			Lemma~\ref{lemma-sequence}~(a)).
			In all cases considered above, the conjecture holds, since
			$\lceil n/2\rceil=3$.

			\item If $n=6$ then there are six non-isomorphic trees, as shown in Figure~\ref{fig:treesn6}:
			\begin{figure}[ht!]
				\centering            \includegraphics[width=0.6\linewidth]{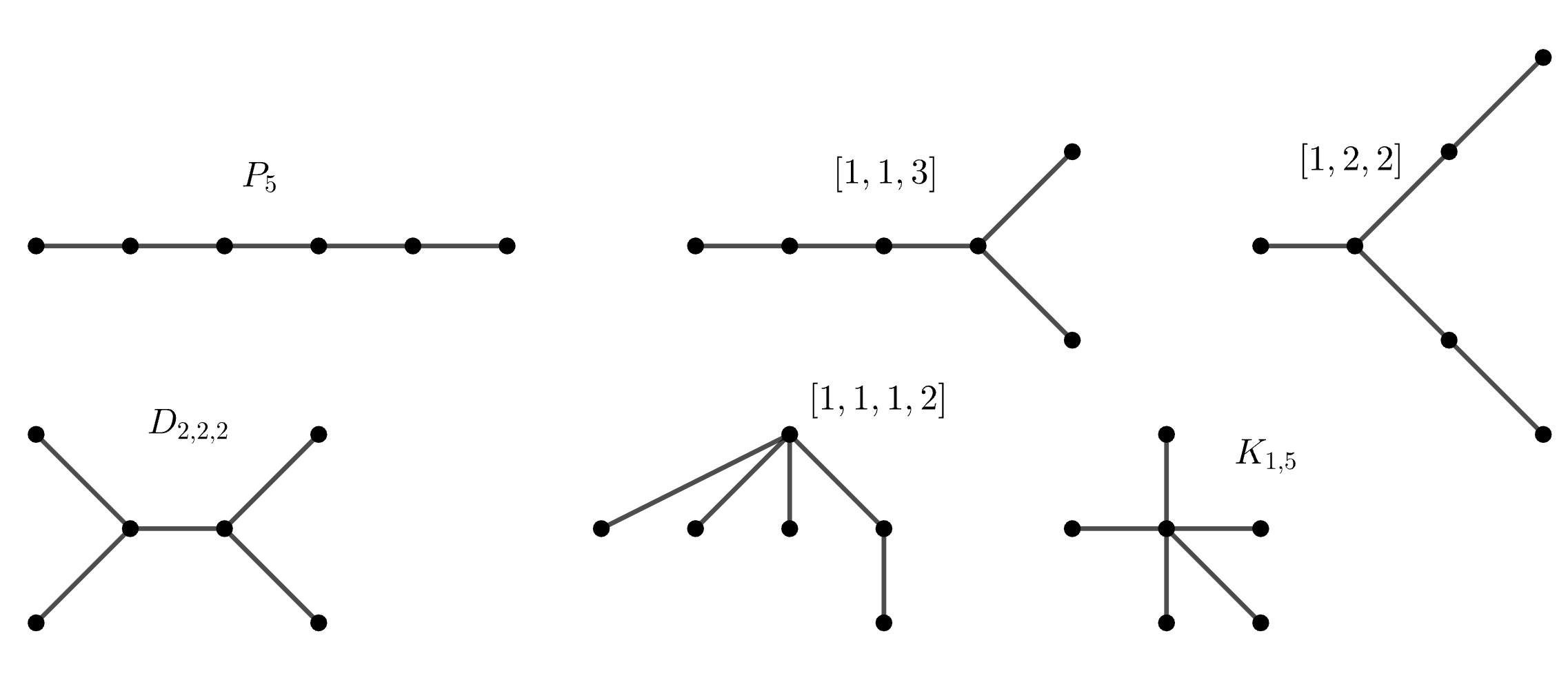}
				\caption{The six non-isomorphic trees on six vertices: the path graph $P_6$, the star $K_{1,5}$, the double-broom tree $D_{2,2,2}$, and the three starlike trees with branch lengths $[1,1,3]$, $[1,2,2]$, and $[1,1,1,2]$.}
				\label{fig:treesn6}
			\end{figure}  
			For $T=P_6$, we choose a central internal vertex as the root and
			apply $\operatorname{Diagonalize}$. In each leaf, we obtain
			$x_1(\alpha)<0$, obtaining two negative entries.  Following the process, we obtain $x_2(\alpha)>0$ counting two positive entries and 
			\[
			x_3(\alpha)
			=2\alpha-d_\alpha-\frac{(1-\alpha)^2}{x_2}
			=\frac{-\alpha(7\alpha^2-18\alpha+9)}
			{3(11\alpha^2-18\alpha+9)}<0.
			\]
			Indeed, if $0<\alpha\leq\frac12$, then
			\[
			7\alpha^2-18\alpha+9
			=7\alpha^2+9(1-2\alpha)>0.
			\]
			Since $-\alpha<0$ and
			$3(11\alpha^2-18\alpha+9)>0$, it follows that
			$x_3(\alpha)<0$.
			For the starlike tree $K_{1,5}$, we choose the central vertex as the root and apply $\operatorname{Diagonalize}$. In each leaf, we obtain
			$x_1(\alpha)<0$, obtaining five negative entries. Analogously, for $T=[1,1,3]$, $T=[1,2,2]$, and $T=[1,1,1,2]$, we choose the root and apply $\operatorname{Diagonalize}$. At each leaf, the algorithm yields $x_1(\alpha)<0$. Thus, we obtain at least three negative entries for each of the first two trees and at least four negative entries for the third. 
			Finally, for  $D_{2,2,2}$ we choose any internal vertex as the root, then we apply Diagonalize obtaining the value $x_1(\alpha)<0$ at each leaf, counting four negative entries. In all cases considered above, the conjecture holds, since $\lceil n/2\rceil=3$.
			
		\end{enumerate}
	\end{proof}

	\section{Narrowing down the interval for the conjecture}\label{sec:narrow}
	
	Before we proceed to prove the Conjecture~\ref{thm:main-conjecture} we provide some examples to narrow down the range of $\alpha$. More precisely, inspired by Lemma~\ref{lemma-sequence} item (c-3) and Corollary~\ref{cor:conjecture_small_n}, we conclude that the conjecture generally fails for $\alpha>1/2$.
	
	\begin{example}\label{ex:counter-example-individual}
		Consider $\alpha:={\frac {193819}{60367}}-{\frac {606\,\sqrt {70433}}{60367}} \approx 0.546512984 > 1/2$. For $n=101$, one obtains $r_0:=21 \in \mathbb{N}$. Now consider $T:=u \oplus P_{1} \ast S_{24} \oplus P_{1} \ast S_{25}$. By construction $b_1(24)>0$ and $b_1(25)>0$. Therefore, even if $a_T(u)<0$ we must have $\sigma_\alpha(T) \geq \lceil\frac{n}{2}\rceil$, that is, it does not satisfy the conjecture.
	\end{example}
	
	The next example shows sequences of values of $\alpha$ arbitrarily close to $\alpha=1/2$ and $\alpha=1$ where the Conjecture~\ref{thm:main-conjecture} fails. We also provide nontrivial cases where $d_{\alpha} \in \operatorname{Spec}(A_{\alpha})$. Therefore, the interval $0\leq \alpha \leq 1/2$ is sharp for the validity of Theorem~\ref{thr:main}.
	\begin{proposition} \label{prop:counterexamples}
		Consider the following families of trees:
		\begin{enumerate}
			\item If $T_n= u + P_{1} \ast S_{k} \oplus P_{1} \ast S_{k+1}$, $k:=\left\lfloor \frac{n}{4} \right\rfloor -1$ (or equivalently $n=4k+5$) and
			\[\alpha_{n}:=\alpha_{n,k}= \left( 1+\sqrt {{\frac { \left( 2n-4 \right)  \left( kn-n+2
						\right) }{{n}^{2} \left( n-2-2\,k \right) }}} \right) ^{-1}.
			\]
			Then $\sigma_\alpha(T_n)=2k+3 > \lfloor n/2\rfloor$ and $\displaystyle\lim_{n\to \infty} \alpha_{n} =1/2$.
			\item If $T_j= u + P_{1} \ast S_{2} \oplus j (P_{1} \ast S_{4})$, $n=6+9j$, $j\geq3$, and
			\[\alpha_{j}:=\alpha_{n,3}= \left( 1+\sqrt {{\frac { \left( 2n-4 \right)  \left( 2n+2
						\right) }{{n}^{2} \left( n-8 \right) }}} \right) ^{-1}.
			\]
			Then $\sigma_\alpha(T_j)\geq 2+5 j > \lfloor n/2\rfloor$ and $\displaystyle \lim_{j\to \infty} \alpha_{j} =1$.

			\item If $T_k=u+k(P_{1}\ast S_{2})$, $k\geq2$, $n=5k+1$, and
			\[\alpha_{n,2}= \left( 1+\sqrt {{\frac { \left( 2n-4 \right)  \left( n+2
						\right) }{{n}^{2} \left( n-6 \right) }}} \right) ^{-1},
			\] then $d_{\alpha_{n,2}} \in \operatorname{Spec}(A_{\alpha_{n,2}})$ with multiplicity $k-1$ and $\displaystyle \lim_{n\to \infty} \alpha_{n,2} =1$.
		\end{enumerate}
	\end{proposition}
	\begin{proof}
		
		According to the hypotheses presented in each item, we have:
		\begin{enumerate}
			\item We combine the formula from Lemma~\ref{lemma-sequence} (c-2), where we choose $k=r_0$ and Lemma~\ref{lemma-sequence} (c-4,5) to verify that $b_1(k)=0$ and $b_1(k+1)>0$.
			\begin{figure}[ht!]
				\centering    \includegraphics[width=0.25\linewidth]{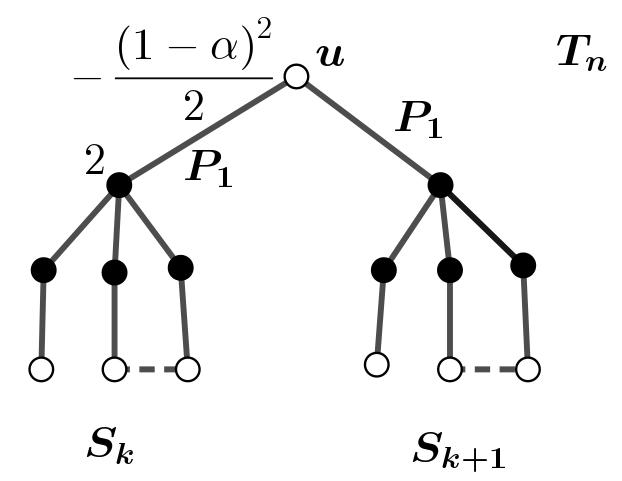}
				\caption{The tree $T_n$ with the vertices entries coming from $\operatorname{Diagonalize}(A_\alpha(T_n),-d_\alpha)$.}
				\label{fig:tree_Tn}
			\end{figure}
			By simply counting the number of positive vertices we obtain $\sigma_\alpha(T_n)=2k+3 > \lfloor n/2\rfloor$.
			Since \(n=4k+5\), we have
			\[
			2n-4=8k+6,\qquad
			kn-n+2=4k^2+k-3,\qquad
			n-2-2k=2k+3.
			\]
			Thus,
			\[
			\alpha_n
			=
			\left(
			1+
			\sqrt{
				\frac{(8k+6)(4k^2+k-3)}
				{(4k+5)^2(2k+3)}
			}
			\right)^{-1}.
			\]
			For $k \longrightarrow \infty$
			\[
			\frac{(8k+6)(4k^2+k-3)}
			{(4k+5)^2(2k+3)}
			\longrightarrow 1.
			\]
			Therefore,
			\[
			\lim_{n\to\infty}\alpha_n
			=
			\lim_{k\to\infty}
			\left(
			1+
			\sqrt{
				\frac{(8k+6)(4k^2+k-3)}
				{(4k+5)^2(2k+3)}
			}
			\right)^{-1}
			=
			\frac{1}{1+1}
			=
			\frac12.
			\]
			\item Similarly to the previous item, we combine the formula from Lemma~\ref{lemma-sequence} (c-2), where we choose $k=3$ and Lemma~\ref{lemma-sequence} (c-4,5) to verify that $b_1(2)<0$ and $b_1(4)>0$.
			\begin{figure}[ht!]
				\centering    \includegraphics[width=0.5\linewidth]{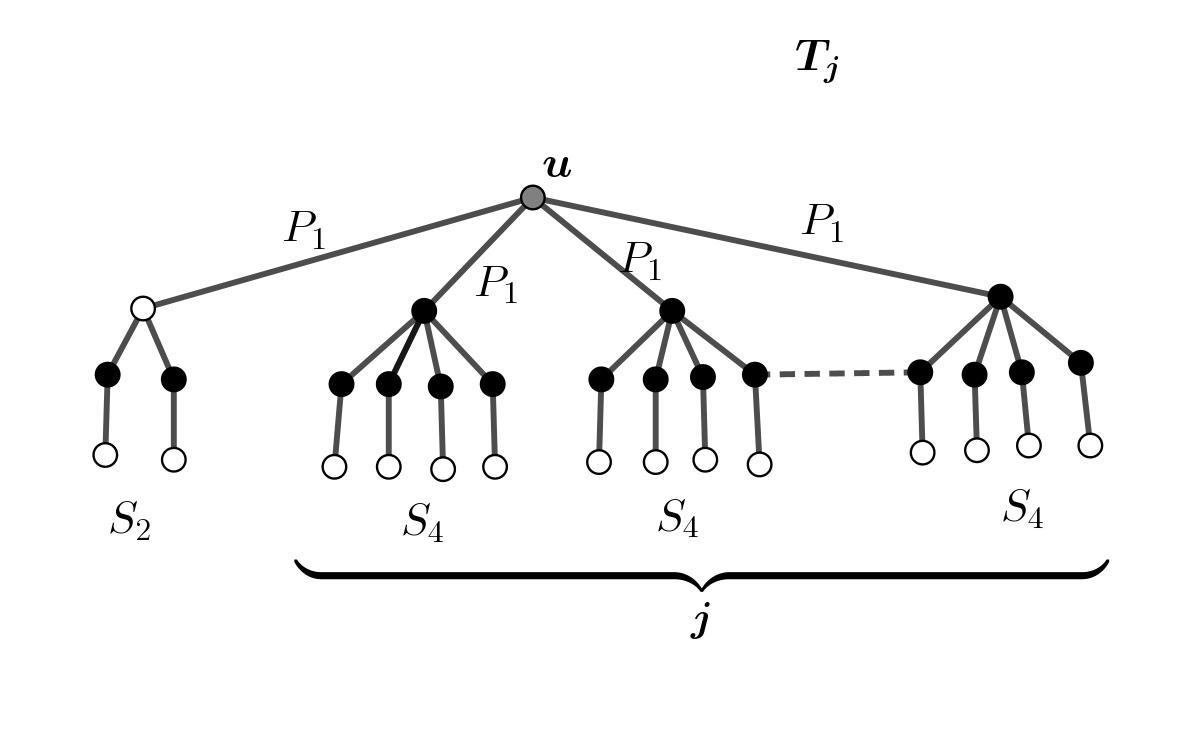}
				\caption{The tree $T_j$ with the vertices entries coming from $\operatorname{Diagonalize}(A_\alpha(T_j),-d_\alpha)$.}
				\label{fig:tree_Tj}
			\end{figure}
			By counting the positive vertices we obtain $\sigma_\alpha(T_j)\geq 2+5 j > \lfloor n/2\rfloor$.
			
			Since \(n=6+9j\), we have \(n\to\infty\) as \(j\to\infty\). Furthermore,
			\[
			\frac{(2n-4)(2n+2)}
			{n^2(n-8)}
			=
			\frac{
				\left(2-\frac{4}{n}\right)
				\left(2+\frac{2}{n}\right)
			}{
				n-8
			}
			\longrightarrow 0.
			\]
			Therefore, $\sqrt{
				\frac{(2n-4)(2n+2)}
				{n^2(n-8)}
			}
			\longrightarrow 0,$ and consequently
			\[
			\lim_{j\to\infty}\alpha_j
			=
			\lim_{j\to\infty}
			\left(
			1+
			\sqrt{
				\frac{(2n-4)(2n+2)}
				{n^2(n-8)}
			}
			\right)^{-1}
			=1.
			\]
			\item Once again, we combine the formula from Lemma~\ref{lemma-sequence} (c-2), where we choose $k=2$ and Lemma~\ref{lemma-sequence} (c-4,5) to verify that $b_1(2)=0$.
			\begin{figure}[ht!]
				\centering    \includegraphics[width=0.5\linewidth]{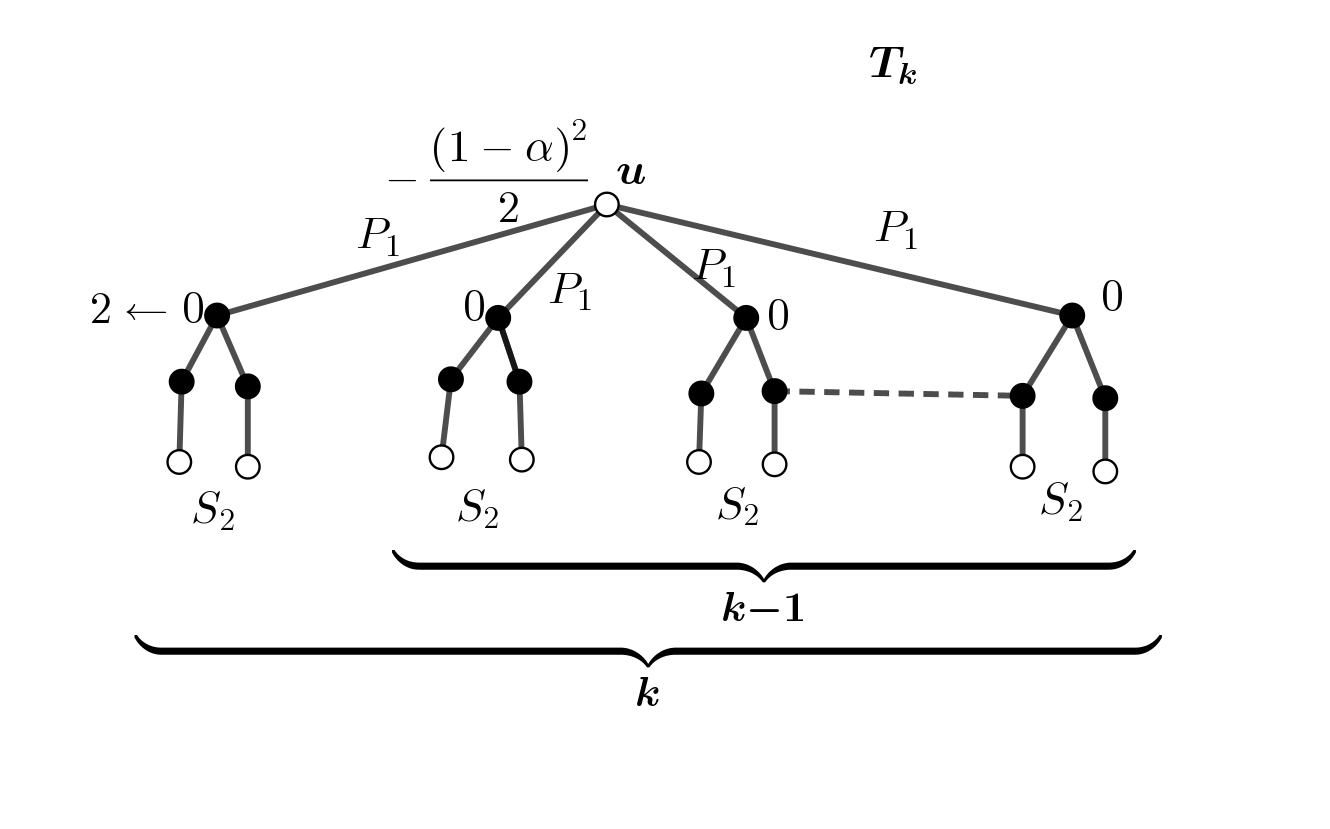}
				\caption{The tree $T_k$ with the vertices entries coming from $\operatorname{Diagonalize}(A_\alpha(T_k),-d_\alpha)$.}
				\label{fig:figmul(k-1)}
			\end{figure}
			Since $\alpha=\alpha_{n,2}$ and applying the algorithm $\operatorname{Diagonalize}$ to each copy of $P_1\ast S_2$
			produces a zero at the same vertex. When the first zero is encountered, the zero replaces the corresponding entries by $2$ and $-\frac{(1-\alpha)^2}{2}$ (see Figure~\ref{fig:figmul(k-1)}). The remaining $k-1$ copies are not affected by this operation and hence their corresponding diagonal entries remain equal to zero. Therefore, the diagonal matrix contains exactly $k-1$ zero entries. Consequently,
			$
			d_{\alpha_{n,2}}
			\in
			\operatorname{Spec}(A_{\alpha_{n,2}})
			$
			with multiplicity $k-1$. Finally, $\frac{(2n-4)(n+2)}
			{n^2(n-6)}
			\longrightarrow 0,$ and consequently $\lim_{n\to\infty}\alpha_{n,2}=1.$
		\end{enumerate}
	\end{proof}
	
	\section{Proper transformations}\label{sec:proper}
	In this section we present proper transformations which are performed on a tree $T$ that preserves the number of vertices and do not decrease the number of $A_{\alpha}$-eigenvalues above the average degree $d_\alpha$. Such transformations are local and for this reason we can translate this property in terms of elementary rational recursions.
	
	We start by analyzing the signs of the vertices after applying \(\operatorname{Diagonalize}(A_\alpha(T),-d_\alpha) \) for a tree having \(r\) pendant
	\(P_2\)'s attached to a path, as shown in Figure~\ref{fig:diagAalfa}.
	\begin{definition} If $T'$ is obtained from $T$ by a transform $\tau$, preserving the number of vertices and the shape of a tree such that $\sigma_\alpha(T')\geq \sigma_\alpha(T) $, where
		\begin{align*}
			\sigma_\alpha(T)=m_T(\alpha)(d_\alpha,\lambda_1(T)]=m_T(\alpha)(d_\alpha,\Delta]. 
		\end{align*}
	\end{definition}

	\begin{proposition}[Star-up transform]\label{prop:starup}  
		Let $u$ be a vertex of a tree $T$ with $n \geq 7$ vertices.
		If $u$ has a gpp $P_q \ast S_r$, $q\geq 2$, $0<\alpha \leq 1/2$ and 
		$0 \leq r \leq r_0$,
		then the transformation $T \stackrel{\tau}{\rightarrow} T'$ given by
		\[
		P_q \ast S_r \;\xrightarrow{\tau}\; P_{q-2} \ast S_{r+1}
		\]
		is proper.
		
		\begin{figure}[ht!]
			\centering    \includegraphics[width=0.55\linewidth]{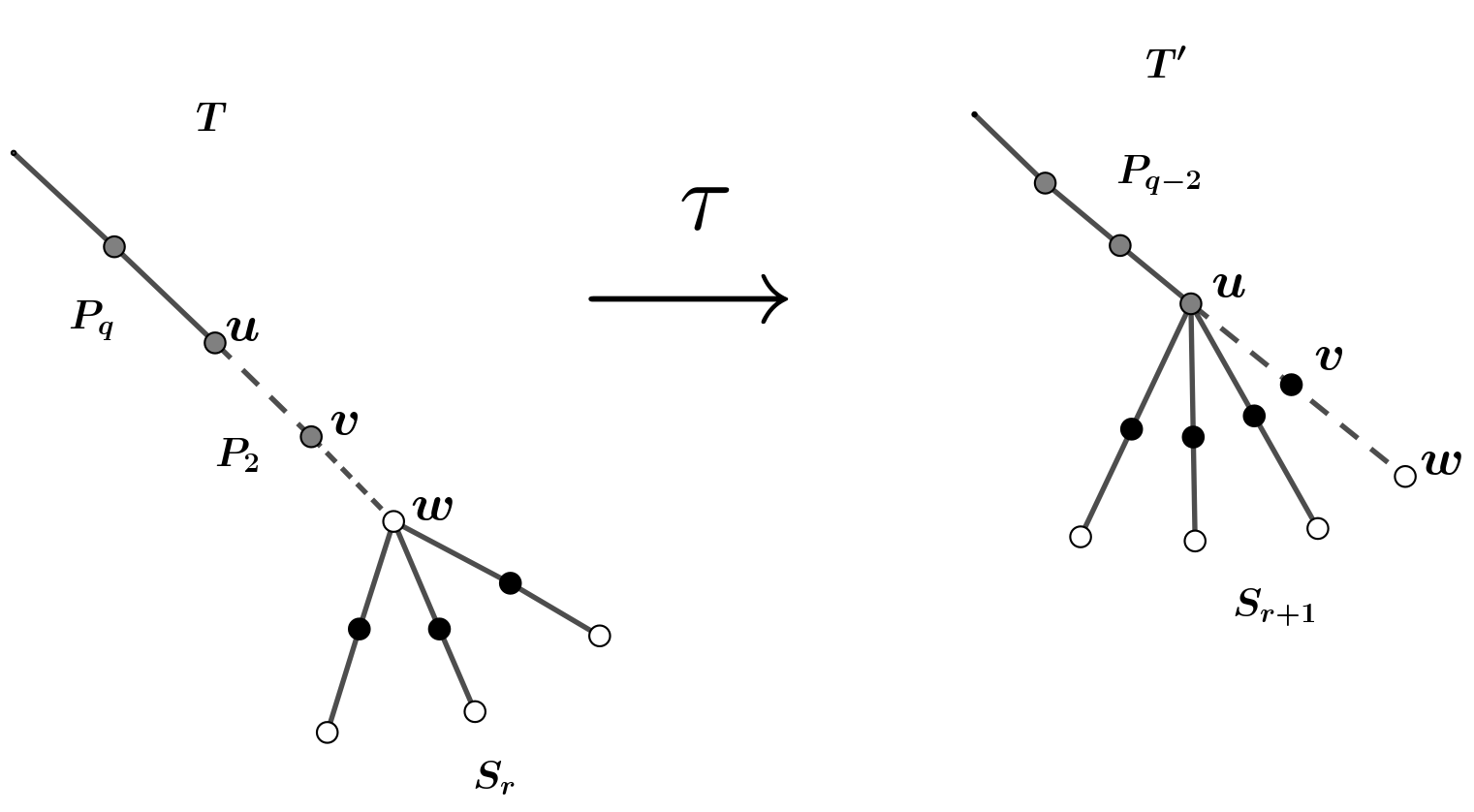}
			\caption{Star-up transform}
			\label{fig:Startran}
		\end{figure}
	\end{proposition}
	
	\begin{proof}
		We consider the transformation on a tree $T$, as illustrated in Figure~\ref{fig:Startran}. 
		This removes one pendant path $P_2$ at the vertex $u$ connected to the star $S_r$
		formed by $r$ paths $P_2$, and produces a new tree $T'$ with a star $S_{r+1}$
		attached at $u$. We consider $u$ as the root of $T$, meaning that it will be the
		last vertex to be processed.

		Using Lemma~\ref{lemma-eigen}, we conclude that at least half of the outputs at the star are positive
		when considering the transformed matrices $\operatorname{Diag}(A_\alpha(T),-d_\alpha)$
		and $\operatorname{Diagonalize}(A_\alpha(T'),-d_\alpha)$.
		
		The parts of both graphs that are invariant under $\tau$ retain the same signs. The only possible changes occur at the root $u$ and along the path $\{v,w\}$. We distinguish three cases.
		\begin{enumerate}
			\item[\textbf{Case 1:}] If $r<r_0$, then $b_1(r)<0$  and $b_2=\frac{2\alpha}{n}-\frac{(1-\alpha)^2}{b_1}>0$.\\
			Hence the signs at $w$ and $v$ do not change. Consequently, the only remaining vertex to analyze is $u$. \\ 
			
			Let $v_1,v_2,\ldots,v_l$ be other children of $u$ in the unchanged part. Then, we have two sub-cases:
			\begin{enumerate}
				\item $a_T(v_j)=0$ for some $j\in\{1,2,\dots,l\}$, the same for $a_{T'}(v_j)=0$, in this situation, $a_T(u)=a_{T'}(u)=-\frac{(1-\alpha)^2}{2}$.
				Therefore,  $\sigma_\alpha(T)=\sigma_\alpha(T')$ and $\tau$ is proper.  
				
				\item Otherwise, if 
				$a_T(v_j)=a_{T'}(v_j)\neq 0,\ \forall j\in\{1,2,\dots,l\}$  then
				\[
				a_T(u)=\alpha \deg_T(u)-d_\alpha-\displaystyle \sum_{j=1}^{l} \frac{(1-\alpha)^2}{a_{T}(v_j)}-\frac{(1-\alpha)^2}{b_2(r)}
				\]
				
				\[
				a_{T'}(u)=\alpha \deg_{T'}(u)-d_\alpha-\displaystyle \sum_{j=1}^{l} \frac{(1-\alpha)^2}{a_{T'}(v_j)}-\frac{(r+1)(1-\alpha)^2}{x_2}
				\]
				Therefore, using the value of $b_1(r)$ and $b_2(r)$ in \eqref{eq:sistema_b} and $x_2$ in \eqref{eq:recorrencia} we obtain
				\begin{align*}
					a_{T'}(u)-a_T(u)
					&= r\alpha-\frac{(r+1)(1-\alpha)^2}{x_2}
					+\frac{(1-\alpha)^2}{b_2(r)}\\
					&=r \left(\alpha-\frac{(1-\alpha)^2}{x_2}\right)
					+\frac{(1-\alpha)^2}{b_2(r)}-\frac{(1-\alpha)^2}{x_2}\\
					&= r \left(\alpha-\frac{(1-\alpha)^2}{x_2}\right)
					+(1-\alpha)^2\left(\frac{1}{b_2(r)}-\frac{1}{x_2}\right)\\
					\shortintertext{Since \(b_1-x_1 = r \left(\alpha-\frac{(1-\alpha)^2}{x_2}\right)\) and \(b_2=\frac{2\alpha}{n}-\frac{(1-\alpha)^2}{b_1}\), we have:}
					a_{T'}(u)-a_T(u)
					&= (b_1-x_1) + (1-\alpha)^2\left(\frac{1}{b_2(r)}-\frac{1}{x_2}\right)\\
					&= \left(b_1+\frac{(1-\alpha)^2}{\frac{2\alpha}{n}-\frac{(1-\alpha)^2}{b_1}}\right)-\left(x_1+\frac{(1-\alpha)^2}{\frac{2\alpha}{n}-\frac{(1-\alpha)^2}{x_1}}\right).
				\end{align*}
				Defining the function $g(t)$ by
				\[
				g(t) = t + \frac{(1-\alpha)^2}{\frac{2\alpha}{n} - \frac{(1-\alpha)^2}{t}}
				= \frac{2\alpha t^2}{2\alpha t - n(1-\alpha)^2},
				\]
				we can write $a_{T'}(u) - a_T(u) = g(b_1) - g(x_1).$ We need to show that $g(t)-g(x)$ is positive for $x_1 \leq t \leq b_1 < 0$.
				
				Since the  asymptote of $g(t)$ is $ t_0 = \frac{(1-\alpha)^2}{2\alpha}\,n,$  with $0<\alpha<1$ and $n\geq 3$, we have $t_0>0$. Moreover, as $b_1<0$, it follows that $b_1 < 0 < t_0,$ so $b_1$ lies strictly on the left of the asymptote.
				
				One can show that $g'(t) = \frac{4\alpha t \bigl(\alpha t - n(1-\alpha)^2\bigr)}{\bigl(2\alpha t - n(1-\alpha)^2\bigr)^2}.$ Since the denominator $\bigl(2\alpha t - n(1-\alpha)^2\bigr)^2$ is always positive, the sign of $g'(t)$ is determined by the numerator. For $t < 0$ (as $t \leq b_1 < 0$), $\alpha > 0$ and $n > 1$, we have:
				\begin{itemize}
					\item $4\alpha t < 0$ (negative, since $t < 0$),
					\item $\alpha t - n(1-\alpha)^2 < 0$ (negative, since $\alpha t < 0$ and $-n(1-\alpha)^2 < 0$).
				\end{itemize}
				Thus, $4\alpha t \cdot \bigl(\alpha t - n(1-\alpha)^2\bigr) > 0$. Therefore, $g(t)$ is strictly increasing in the desired interval.
				
				As $x_1 \leq b_1$, we conclude that $g(x_1) \leq g(b_1)$ therefore, $a_{T'}(u) \geq  a_T(u) $
			\end{enumerate}

			\item[\textbf{Case 2:}] If $r=r_0 \in \mathbb{N}$, then $b_1(r)=0$.
			\begin{figure}[ht!]
				\centering    \includegraphics[width=0.45\linewidth]{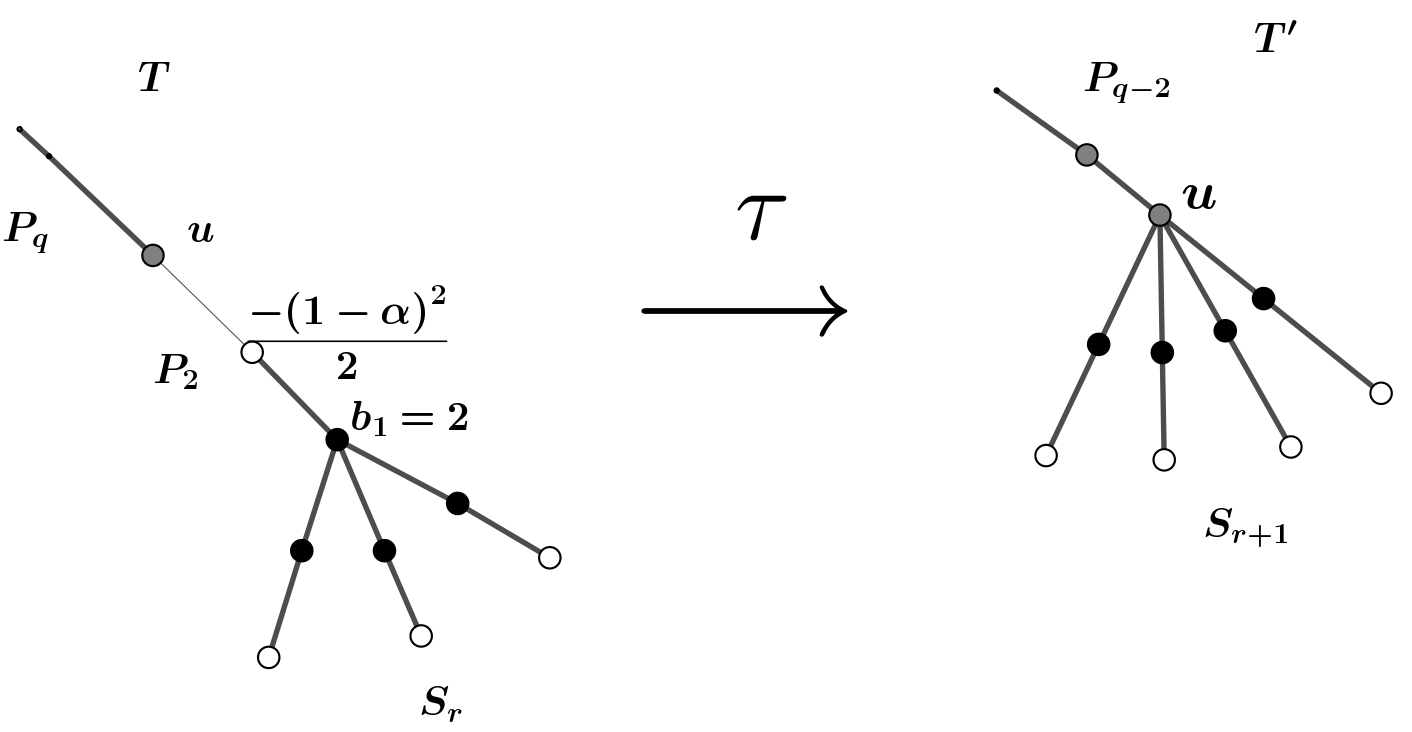}   
				\caption{Star-up transform for $b_1=0$}    \label{fig:Transb1}
			\end{figure}
			Since the balance of signal $+$ and $-$ remains the same in the pendant path, we only need to track the root $u$.
			
			Let $v_1, v_2, \ldots, v_l$ ($l \geq 1$) be other children of $u$ in the unchanged part. Then we have two cases:
			\begin{enumerate}
				\item $a_T(v_j)=0$ for some $j\in\{1,2,\dots,l\}$, the same for $a_{T'}(v_j)=0$, in this situation, $a_T(u)=a_{T'}(u)=-\frac{(1-\alpha)^2}{2}$.
				Therefore  $\sigma_\alpha(T)=\sigma_\alpha(T')$ and $\tau$ is proper.
				\item Otherwise, if 
				$a_T(v_j)=a_{T'}(v_j)\neq 0,\ \forall j\in\{1,2,\dots,l\}$ then we have
			\end{enumerate}  
			\begin{align*}
				a_T(u)=\alpha(l+1)-d_\alpha-\displaystyle \sum_{j=1}^{l} \frac{(1-\alpha)^2}{a_{T}(v_j)}
			\end{align*}
			\begin{align*}    
				a_{T'}(u)=\alpha(r+l+1)-d_\alpha-\displaystyle \sum_{j=1}^{l} \frac{(1-\alpha)^2}{a_{T'}(v_j)}- \frac{(r+1)(1-\alpha)^2}{x_2}.
			\end{align*}
			
			We claim that $a_{T'}(u)\geq a_{T}(u)$. In order to see that we consider
			\begin{align*}
				a_{T'}(u)-a_{T}(u)=\alpha r-\frac{(r+1)(1-\alpha)^2}{x_2}
			\end{align*}
			or equivalently
			\begin{align*}
				a_{T'}(u)-a_{T}(u)
				&= b_1-\left(x_1+\frac{(1-\alpha)^2}{\frac{2\alpha}{n}-\frac{(1-\alpha)^2}{x_1}}\right) = -\left(x_1+\frac{(1-\alpha)^2}{\frac{2\alpha}{n}-\frac{(1-\alpha)^2}{x_1}}\right).
			\end{align*}
			Consider the function $f(x_1) = -\left(x_1 + \frac{(1-\alpha)^2}{\frac{2\alpha}{n} - \frac{(1-\alpha)^2}{x_1}}\right) $. One can write
			
			\begin{align*}
				f(x_1) 
				&= -\frac{\frac{2\alpha x_1}{n}}{\frac{2\alpha x_1 - n(1-\alpha)^2}{nx_1}}
				= -\frac{2\alpha x_1}{n} \cdot \frac{nx_1}{2\alpha x_1 - n(1-\alpha)^2} \\[4pt]
				&= -\frac{2\alpha x_1^2}{2\alpha x_1 - n(1-\alpha)^2}
				= \frac{2\alpha x_1^2}{n(1-\alpha)^2 - 2\alpha x_1}.
			\end{align*}
			As \( \alpha \in (0,1/2] \) , \( n \ge 3 \) and
			\( x_1 < 0 \), then \( -2\alpha x_1 > 0 \), so $n(1-\alpha)^2 - 2\alpha t 
			= n(1-\alpha)^2 + (-2\alpha t) > 0$. Hence $f(x_1) > 0$ and $a_{T'}(u)>a_{T}(u)$.
			
		\end{enumerate}
		
	\end{proof}
	
	\begin{proposition}\label{prop:stardown}(Star-down transform) 
		Let $u$ be a vertex of  a tree $T$ with $n \geq 7$ vertices. Consider the transformation $T \stackrel{\tau}{\rightarrow} T'$, where the gpps are attached to a vertex $u$,
		$$ P_{1}*S_{r} \oplus  P_{2}  \to  P_{1} \oplus S_{r+1}.$$ If $0\leq r\leq r_0$ and $0<\alpha\leq1/2$, then $\tau$ is a proper transformation.
		
		\begin{proof}Suppose $r< r_0$, then $b_1(r)<0$. In $T'$ the sign of $b_1(r+1)$ can be negative, zero or positive. If it is zero or positive then $T'$ produces an additional positive output so that $\tau$ is proper despite the signal at $u$. Therefore, the remaining case is $b_1(r+1)<0$; then the number of negative and positive outputs remains the same, except possibly at $u$ (see Figure~\ref{fig:stardown}).
			
			\begin{figure}[ht!]
				\centering    \includegraphics[width=0.6\linewidth]{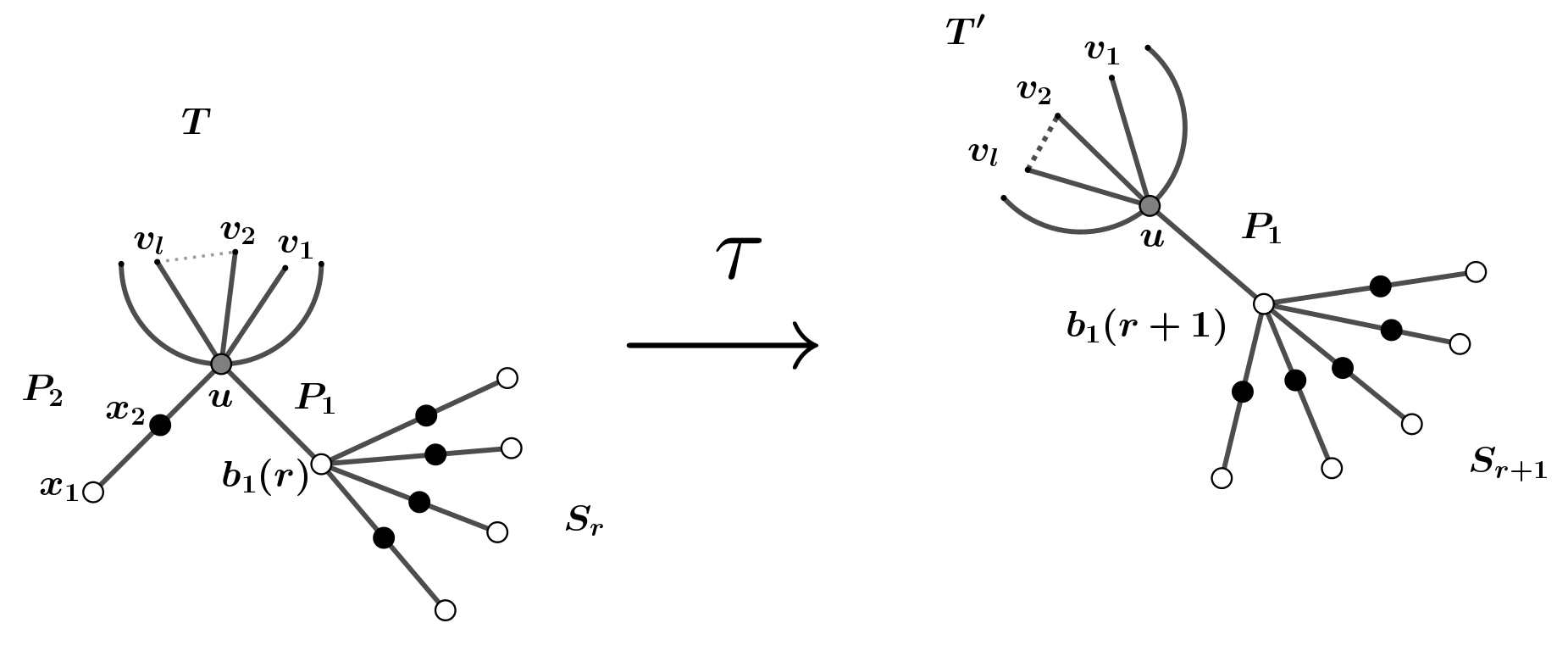}
				\caption{Star-down for $r<r_0$}   \label{fig:stardown}
			\end{figure}
			For $u$ we must consider the possibility of any of his children $v_1,v_2,\ldots,v_l$ to produce zero.
			\begin{enumerate}
				\item $a_T(v_j)=0$ for some $j\in\{1,2,\dots,l\}$, the same for $a_{T'}(v_j)=0$, by the $\operatorname{Diagonalize}(A_\alpha(T),-d_\alpha)$ algorithm, we obtain $a_T(u)=a_{T'}(u)=-\frac{(1-\alpha)^2}{2}$. Therefore  $\sigma_\alpha(T)=\sigma_\alpha(T')$ and $\tau$ is proper.
				\item Otherwise, if 
				$a_T(v_j)=a_{T'}(v_j)\neq 0,\ \forall j\in\{1,2,\dots,l\}$.
				
				We must compare $a_T(u)$ and $a_{T'}(u)$:
			\end{enumerate} 
			\begin{align*}
				a_T(u)=\alpha \deg_T(u)-d_\alpha-\frac{(1-\alpha)^2}{b_1(r)}-\frac{(1-\alpha)^2}{x_2}-\displaystyle \sum_{j=1}^{l} \frac{(1-\alpha)^2}{a_{T}(v_j)}
			\end{align*}
			\begin{align*}
				a_{T'}(u)=\alpha \deg_{T'}(u)-d_\alpha-\frac{(1-\alpha)^2}{b_1(r+1)}-\displaystyle \sum_{j=1}^{l} \frac{(1-\alpha)^2}{a_{T'}(v_j)}
			\end{align*}
			
			We claim $a_{T'}(u)\geq a_{T}(u)$
			\begin{align*}
				a_{T'}(u)-a_{T}(u)=-\left(\alpha- \frac{(1-\alpha)^2}{x_2}\right)+(1-\alpha)^2\left(\frac{1}{b_1(r)}-\frac{1}{b_1(r+1)}\right)
			\end{align*}
			As $b_1(r+1)-b_1(r)=\alpha-\frac{(1-\alpha)^2}{x_2}$. Then
			\begin{align*}
				a_{T'}(u)-a_{T}(u)=\left(\alpha- \frac{(1-\alpha)^2}{x_2}\right)\left(-1+\frac{(1-\alpha)^2}{b_1(r)b_1(r+1)}\right)
			\end{align*}
			From Lemma~\ref{lemma-sequence} (b) we recall that $m(\alpha)=\alpha - \frac{(1-\alpha)^2}{x_2} > 0$. 
			
			We claim that
			\[
			-1 + \frac{(1-\alpha)^2}{b_1(r)\,b_1(r+1)} > 0.
			\]
			Since $b_1(r) < b_1(r+1) < 0$, we have $b_1(r)\,b_1(r+1) > 0$, and the inequality above is equivalent to
			
			\begin{equation} \label{eq:ggf}
				(1-\alpha)^2 > b_1(r)\,b_1(r+1).
			\end{equation}
			By Lemma~\ref{lemma-sequence} (c-4) we have $-\alpha < b_1(r) < b_1(r+1) < 0$.
			Note that $\alpha >- b_1(r)>0$ and $\alpha >- b_1(r+1)>0$ thus $b_1(r)\,b_1(r+1) < \alpha^2$. On the other hand, we claim that $\alpha^2\leq (1-\alpha)^2$. Indeed, 
			\[\alpha^2\leq (1-\alpha)^2 \iff \alpha^2\leq 1-2 \alpha+ \alpha^2  \iff 0\leq 1-2 \alpha \iff \alpha \leq 1/2.\]
			Using both inequalities one obtains $b_1(r)\,b_1(r+1) < (1-\alpha)^2$ showing that Equation~\ref{eq:ggf} holds.
			
			Finally, if $r=r_0$ (so $r_0\in\mathbb{N}$), then $b_1(r)=b_1(r_0)=0$ and consequently $a_T(v)=2>0$, $a_{T'}(v)>0$, because $r+1>r_0$. Also, $a_T(u)=\frac{-(1-\alpha)^2}{2}<0$ implying that the transformation is proper, since $\sigma_\alpha$ cannot decrease.
			
			\begin{figure}[ht!]
				\centering    \includegraphics[width=0.6\linewidth]{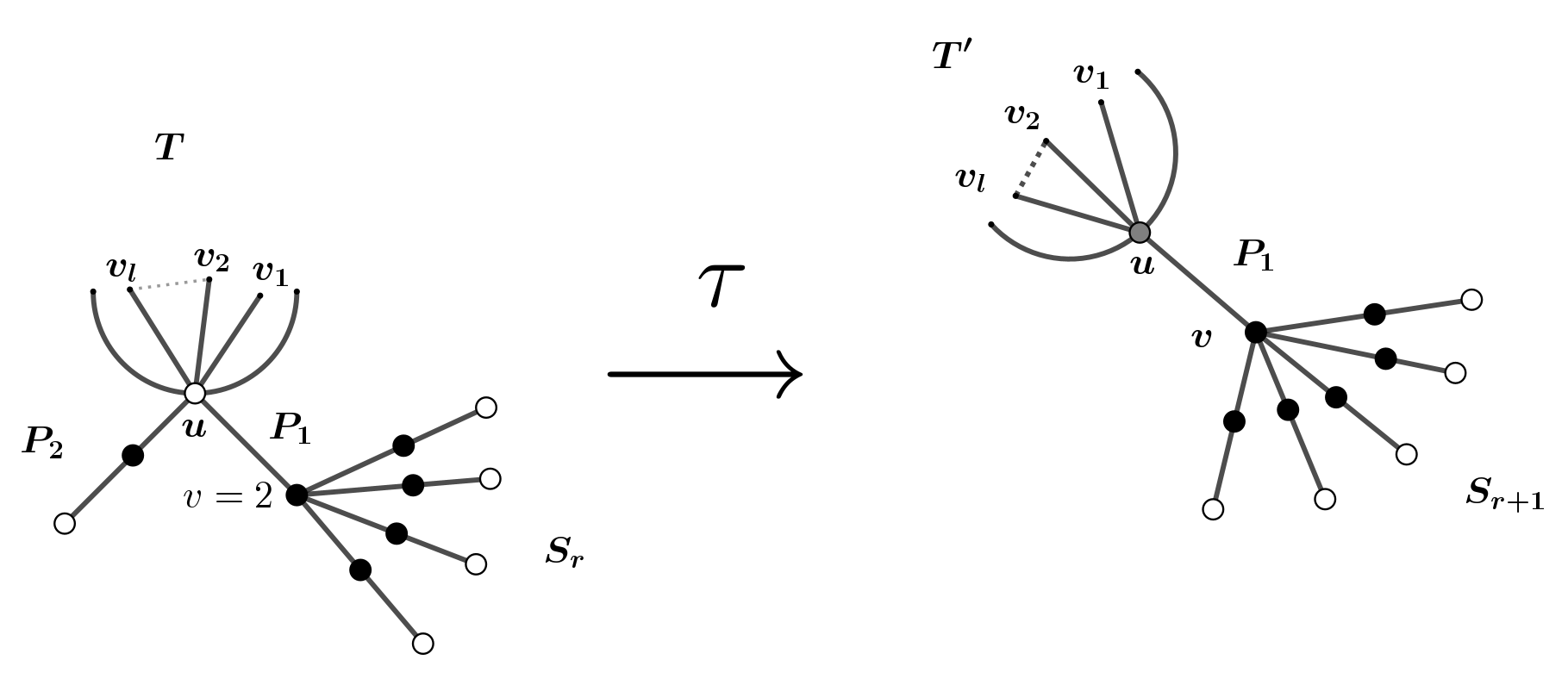}
				\caption{Star-down  for $r=r_0$}   \label{fig:stardown2}
			\end{figure}
		\end{proof}
	\end{proposition}
	
	\begin{proposition}\label{prop:stardown-transf}(Star-star transform)
		Let $u$ be a starlike vertex of a tree $T$ with $n\geq 7$ vertices. Let
		$q_1, q_2 \in \{ 0,1 \}$ and $0<\alpha\leq1/2$. Suppose $u$ has two generalized pendant paths
		$P_{q_1}\ast S_{r_1}$ and $P_{q_2}\ast S_{r_2}$. Then the following
		transformations $T \stackrel{\tau}{\to} T'$ are proper:
		\begin{enumerate} 
			\item[\text{a)}] For $q_1=q_2=0$ and any $r_1', r_2'$ such that $r_1' + r_2' = r_1 + r_2$;
			\[
			(P_0 * S_{r_1}) \oplus (P_0 * S_{r_2}) \stackrel{\tau}{\to} (P_0 * S_{r_1'}) \oplus (P_0 * S_{r_2'}),
			\]
			
			\item[\text{b)}] For $q_1 = q_2 = 1$ and $0 \le r_1, r_2 \le \lfloor r_0 \rfloor$;
			\[
			(P_1 * S_{r_1}) \oplus (P_1 * S_{r_2}) \stackrel{\tau}{\to} 
			\begin{cases}
				P_2 * S_{r_1+r_2} & \text{if } r_1+r_2 \le  \lfloor r_0 \rfloor \\
				P_0 * S_{r_1+r_2-   \lfloor r_0 \rfloor} \oplus (P_2 * S_{\lfloor r_0 \rfloor}) & \text{if } r_1+r_2 >   \lfloor r_0 \rfloor
			\end{cases}
			\]

			\item[\text{c)}] For $0<\alpha\leq 1/2$, $q_1 = 1, q_2 = 0$ and $0 \le r_1, r_2 \le \lfloor r_0 \rfloor$;
			\[
			(P_1 * S_{r_1}) \oplus (P_0 * S_{r_2}) \stackrel{\tau}{\to} 
			\begin{cases}
				P_1 * S_{r_1+r_2} & \text{if } r_1+r_2 \le \lfloor r_0 \rfloor \\
				P_0 * S_{r_1+r_2-\lfloor r_0 \rfloor} \oplus (P_1 * S_{\lfloor r_0 \rfloor}) & \text{if } r_1+r_2 > \lfloor r_0 \rfloor
			\end{cases}
			\]
		\end{enumerate}
	\end{proposition}
	
	\begin{proof}
		We now proceed with the proof of each item.
		\begin{enumerate}
			\item[(a)] 
			\begin{figure}[ht!]
				\centering
				\includegraphics[width=0.6\linewidth]{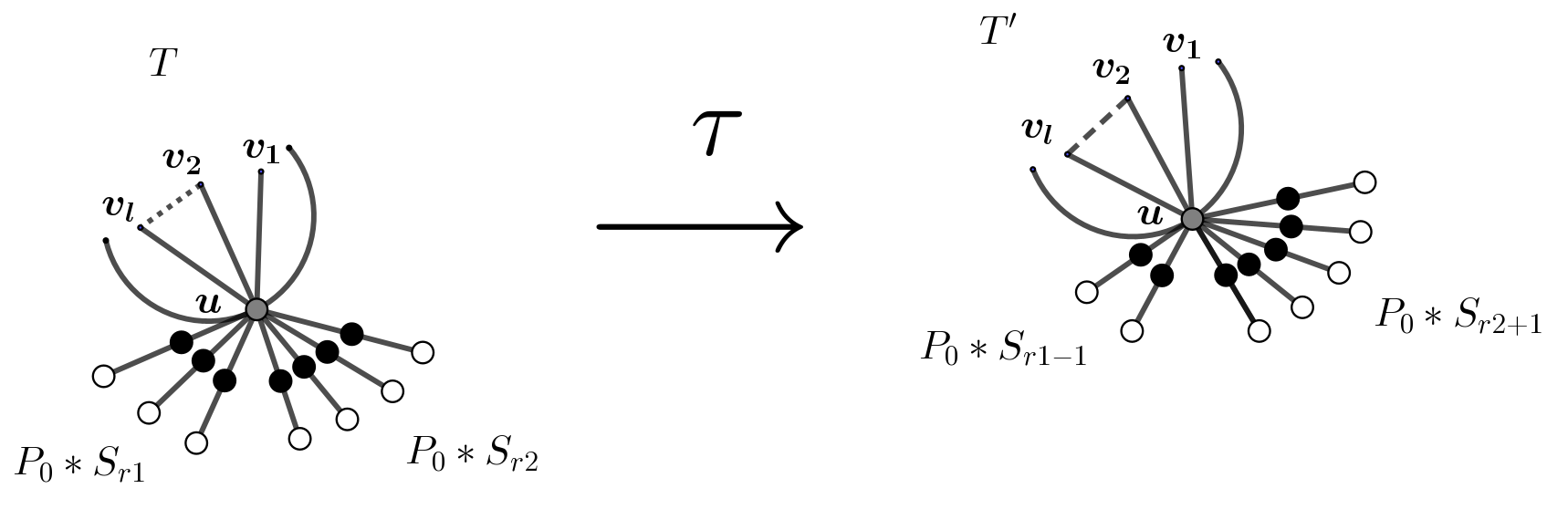}
				\caption{An example of a trivial 0-0 star-star transform}\label{star-star:fig1}
			\end{figure}
			The case $q_{1}=0$ and $q_{2}=0$ (see Figure~\ref{star-star:fig1}) is
			actually a formal rearrangement. We do not change the graph, only perform a
			different partition of the $r_1+r_2$ original $P_2$'s attached to $u$. In
			this case it is not necessary to consider the sign of any vertices. It is therefore a trivial Star-star transformation.
			
			\item[(b)] The case $q_{1}=1$ and $q_{2}=1$ (see Figure~\ref{star-star:fig2}) must be divided into all possibilities.
			
			If $r_1+r_2 \leq r_0$ then the number of positive and negative signs is the same. However, the vertices $v,w,u,z_1,z_2$ in Figure~\ref{star-star:fig2}  must be counted. We consider two cases:
			
			\noindent\textbf{Case 1:} \(r_1+r_2<r_0\)
			
			In this case, by Lemma~\ref{lemma-sequence} (see Figure~\ref{star-star:fig2}) $a_T(v)=b_1(r_1)<0$ and $a_T(w)=b_1(r_2)<0$.
			On the other hand $a_{T'}(z_1)=b_1(r_1+r_2)<0$, and  $a_{T'}(z_2)=b_2(r_1+r_2)>0$. Consequently, the transformation is proper, since $u$ can only decrease  $\sigma_\alpha$ by one in the worst case.

			\noindent\textbf{Case 2:} \(r_1+r_2= r_0\)
			
			In this case, if $r_1$ or $r_2$ are zero, then we assume, without loss of generality, $r_1=0$ and $r_2=r_0$. Therefore, $a_T(v_1) = x_1 < 0,$ $a_T(w) = 2$ (because $b_1(r_0)=0$) and  $a_T(u) = -\frac{(1-\alpha)^2}{2}<0.$
			
			On the other hand $a_{T'}(z_1)=2$ (because $b_1(r_1+r_2)=0$), and $a_{T'}(z_2)= -\frac{(1-\alpha)^2}{2}<0$. Thus $\sigma_\alpha$ does not decrease regardless of the sign of $a_{T'}(u)$.
			
			Finally, if $r_1>0$ and $r_2>0$ then $r_1<r_0$ and $r_2<r_0$, consequently $a_T(v) = b_1(r_1) < 0$ and $a_T(w) = b_1(r_2) < 0.$
			Again, $a_T(z_1)=2>0$, which shows that the transformation is proper regardless of the sign of $a_T(u)$ or $a_{T'}(u)$.
			\begin{figure}[ht!]
				\centering
				\includegraphics[width=0.6\textwidth]{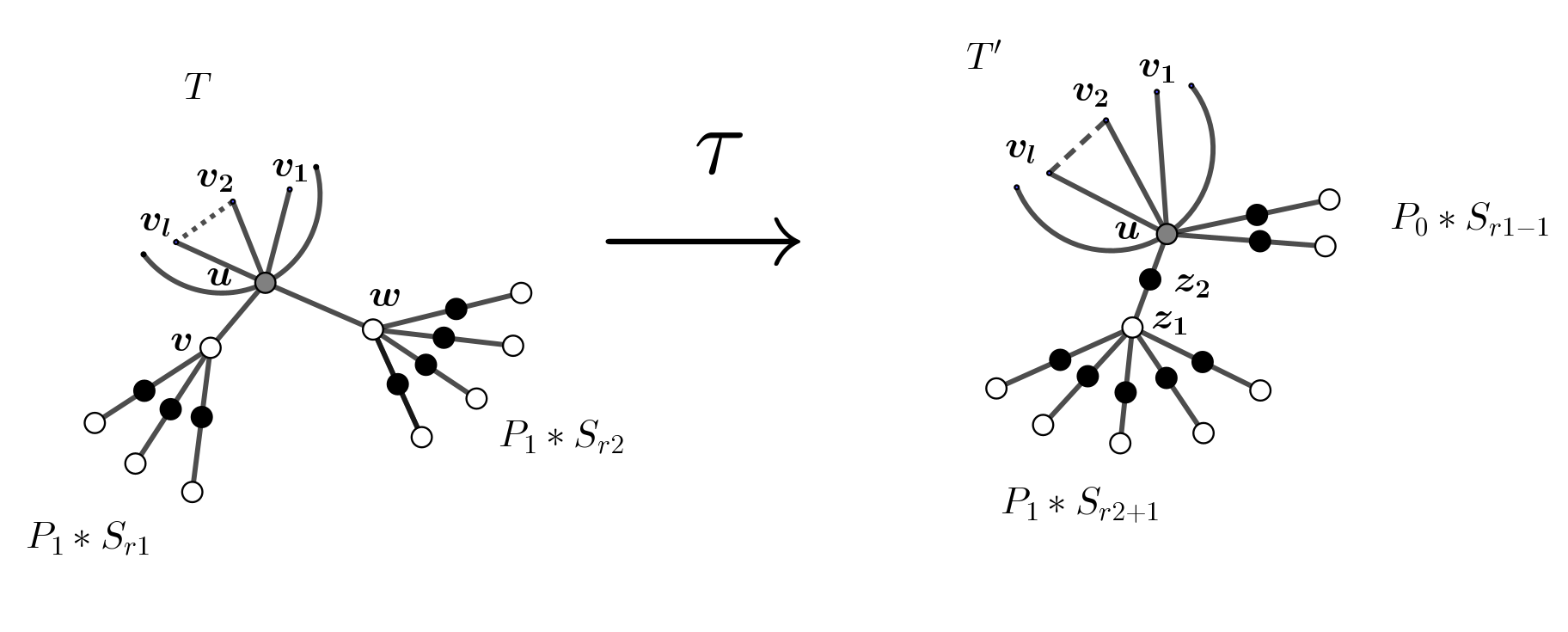}
				\vspace{1em}
				\includegraphics[width=0.5\textwidth]{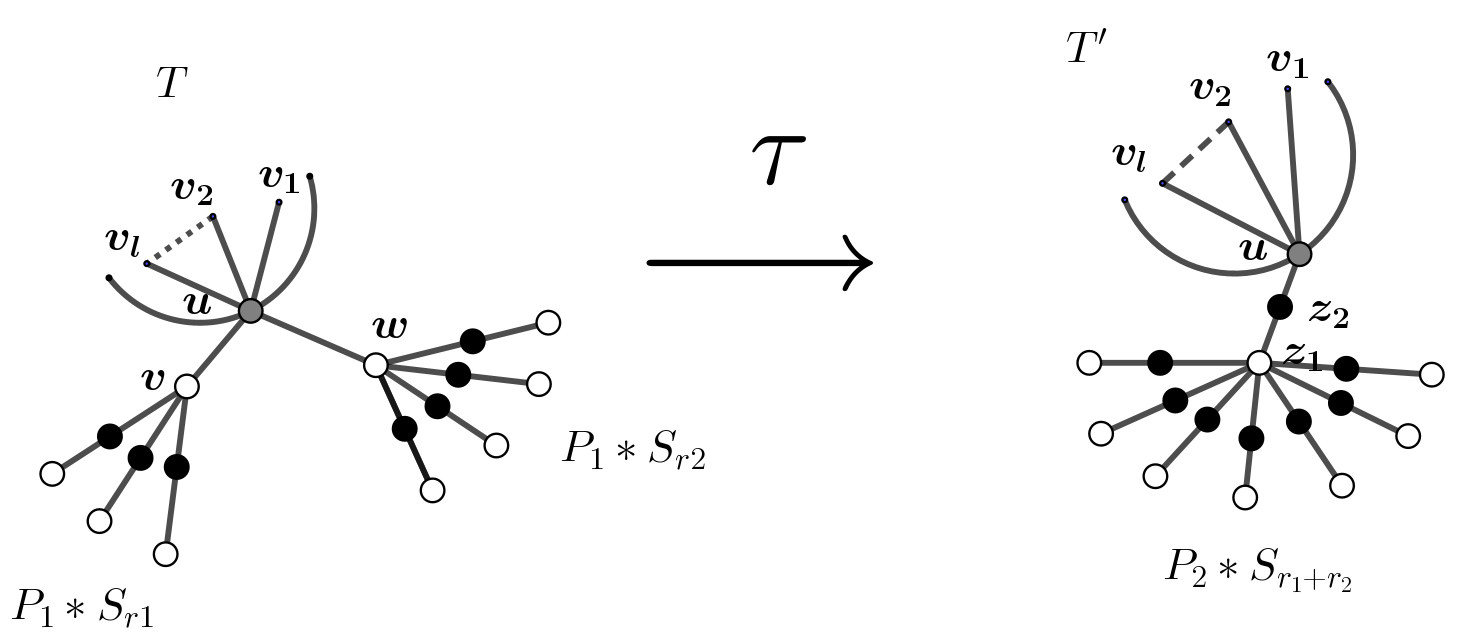}
				\caption{An example of a 1-1 Star-star transform.}
				\label{star-star:fig2}
			\end{figure}
			
			Now assume $r_1+r_2>r_0$ and $r_1\leq r_0, r_2\leq r_0$. We must again consider two cases:
			\[
			\begin{aligned}
				\textbf{Case 1:} \quad & r_0 \in \mathbb{N}_0.\\
				&\text{Obviously, we cannot have } r_1=0 \text{ or } r_2=0.\\
				&\text{Therefore, it is enough to consider the following three sub-cases:}
				\\[6pt]
				&\quad\textbf{Sub-case i:}\quad
				r_1<r_0 \text{ and } r_2<r_0.\\
				&\qquad \text{Then } 
				a_T(v)<0,\quad a_T(w)<0,\quad
				a_{T'}(z_1)>0,\quad a_{T'}(z_2)<0.\\
				&\qquad \text{Thus, } \tau \text{ is proper regardless of the sign of } u,\\
				&\qquad \text{since we obtain an additional positive sign.}
				\\[6pt]
				&\quad\textbf{Sub-case ii:}\quad
				r_1<r_0 \text{ and } r_2=r_0,\\
				&\qquad \text{or, equivalently, } r_1=r_0 \text{ and } r_2<r_0.\\
				&\qquad \text{Then }
				a_T(v)<0,\quad a_T(w)>0,\quad a_T(u)<0,\\
				&\qquad \text{and }
				a_{T'}(z_1)>0,\quad a_{T'}(z_2)<0.\\
				&\qquad \text{Thus, the transformation is proper regardless of the sign of } u
				\text{ in } T'.
				\\[6pt]
				&\quad\textbf{Sub-case iii:}\quad
				r_1=r_2=r_0.\\
				&\qquad \text{Then }
				a_T(v)=0,\quad a_T(w)>0,\quad a_T(u)<0.\\
				&\qquad \text{By the above reasoning, we conclude that } \tau
				\text{ is proper.}
				\\[8pt]
				\textbf{Case 2:}\quad
				& r_0 \notin \mathbb{N}.\\
				& \text{In this case, } r_1,r_2 \leq \lfloor r_0 \rfloor,
				\text{ and hence } r_1<r_0 \text{ and } r_2<r_0.\\
				& \text{Therefore, the reasoning from Sub-case 1-i applies, completing the proof.}
			\end{aligned}
			\]
			\item[(c)]
			The case where $r_1,r_2\leq\lfloor r_0 \rfloor$ and $r_1+r_2\leq \lfloor r_0 \rfloor$ for  $q_{1}=1$ and $q_{2}=0$ (see Figure~\ref{star-star:fig3})  is
			actually a particular application of several proper transformations that we
			call a Star-down transformation in Proposition~\ref{prop:stardown}. In each step,
			a $P_2$ from the $P_0\ast
			S_{r_2}$ is brought down to the path $P_1\ast S_{r_1}$. Thus, the Star-star
			transformation is proper.
			
			The case where $r_1,r_2\leq\lfloor r_0 \rfloor$, but $r_1+r_2> \lfloor r_0 \rfloor$ we can perform Star-down transformations to move a suitable number of
			$P_2$'s so that, after these transformations, $S_{r_2}$ is
			filled to $S_{\lfloor r_0 \rfloor}$, while the star $S_{r_1}$ is left
			with exactly $
			S_{r_1+r_2-\lfloor r_0 \rfloor}$. This completes the proof.
		\end{enumerate}

		\begin{figure}[ht!]
			\centering
			\includegraphics[width=0.5\textwidth]{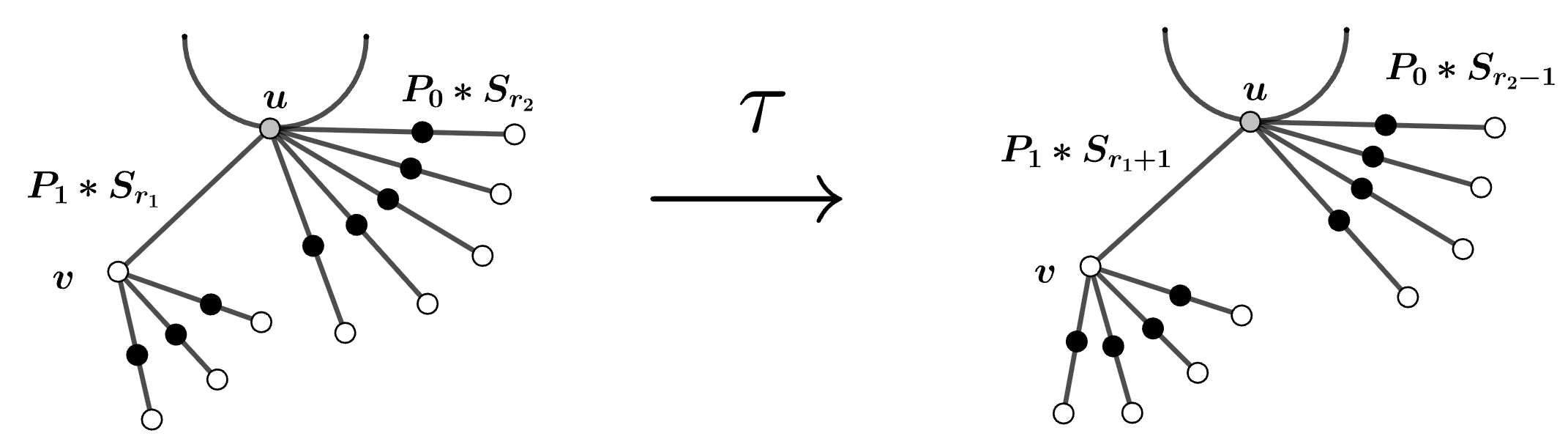}
			\vspace{1em}
			\includegraphics[width=0.5\textwidth]{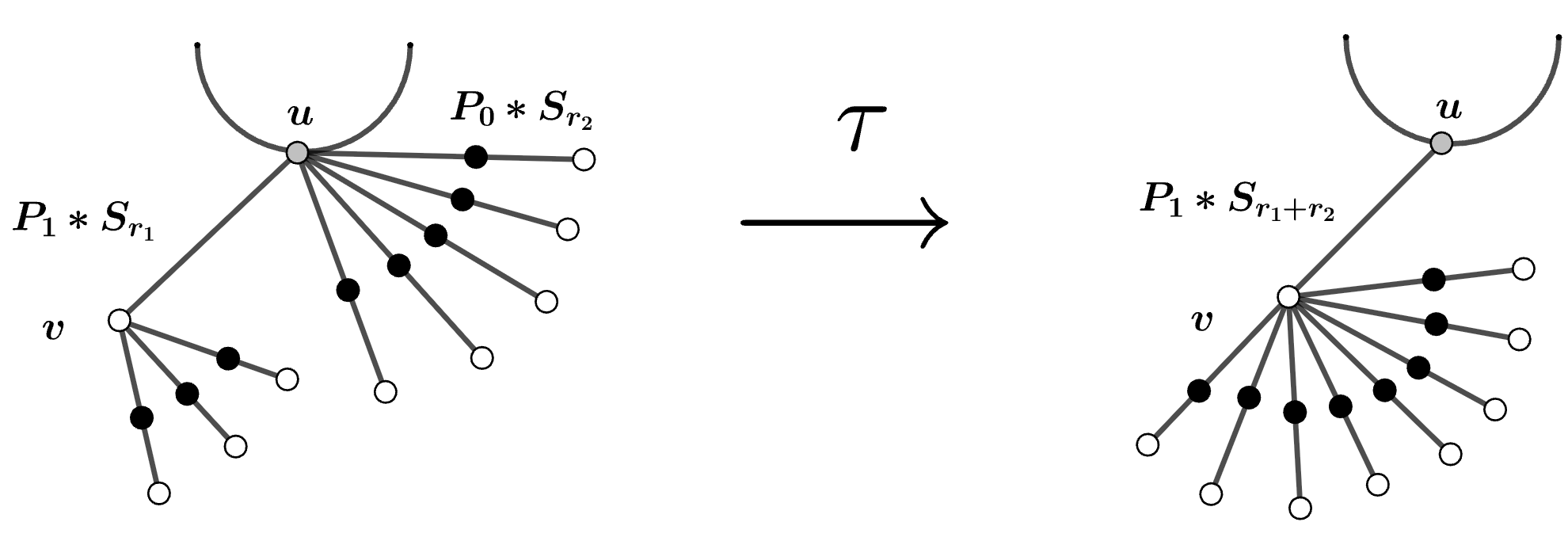}
			\caption{An example of a 1-0 Star-star transform.}
			\label{star-star:fig3}
		\end{figure}    
		
	\end{proof}

	\section{Prototype trees} \label{sec:prototype-trees}
	We now introduce some particular families of trees that satisfy the Conjecture~\ref{thm:main-conjecture}.
	\begin{definition}\label{def:prot}
		Let $j\geq1$ and $\theta\in\{0,1,2,3\}$. For $n=4j+\theta\geq7$,
		define the tree $T_\theta$ of order $n$ by
		\begin{enumerate}
			\item $T_0=u+P_0\ast S_{j-1}\oplus P_1\ast S_j$;
			\item $T_1=u+P_0\ast S_j\oplus P_0\ast S_j$;
			\item $T_2=u+P_0\ast S_j\oplus P_1\ast S_j$;
			\item $T_3=u+P_1\ast S_j\oplus P_1\ast S_j$.
		\end{enumerate}
		For $n=4j+3$, we also define the additional prototype
		$T_{3'}=u+P_0\ast S_j\oplus P_0\ast S_{j+1}$.
	\end{definition}
	Hence, if $T$ is properly transformed into $T_\theta$, then $\sigma_\alpha(T) \leq
	\sigma_\alpha(T_\theta)$. So in order to prove the conjecture, it remains to prove that $\sigma_\alpha(T_\theta) \leq \lfloor\frac{n}{2}\rfloor$ for each one. We now prove that $T_\theta$, in Definition~\ref{def:prot}, satisfies the
	equality in Theorem \ref{thr:main}, that is $\sigma_\alpha(T_\theta) =
	\lfloor\frac{n}{2}\rfloor$.
	
	The next theorem is similar to the corresponding result in \cite{Jacobs2021}; however, we must provide a new proof, since it depends on $\alpha$. We also notice that the introduction of the prototype $T_{3'}$ is necessary to use \cite[Theorem 7.2]{Jacobs2021} without a new proof, since in \cite{Jacobs2021} it is necessary to prove that the tree $T_{3'}$ can be properly transformed into $T_3$.
	\begin{theorem}\label{thr:Talpha}
		Let $j\geq1$ and $\theta\in\{0,1,2,3\}$. For $n=4j+\theta\geq7$, $0<\alpha\leq\frac12$, and
		$d_\alpha=\alpha(2-\frac{2}{n})$, the prototypes $T_\theta$ satisfy
		\begin{enumerate}
			\item[(a)] ${\displaystyle m_{T_0}(\alpha)(-\infty, d_\alpha] = 2j}$;
			\item[(b)] ${\displaystyle m_{T_1}(\alpha)(-\infty, d_\alpha] = 2j+1}$;
			\item[(c)] ${\displaystyle m_{T_2}(\alpha)(-\infty, d_\alpha] = 2j+1}$;
			\item[(d)] ${\displaystyle m_{T_3}(\alpha)(-\infty, d_\alpha] =m_{T_{3'}}(\alpha)(-\infty, d_\alpha] = 2j+2}$.
		\end{enumerate}
		In particular ${\displaystyle m_{T_\theta}(\alpha)(-\infty, d_\alpha] = \left\lceil
			\frac{n}{2} \right\rceil}$, or equivalently, ${\displaystyle\sigma_\alpha(T_\theta)
			= \left \lfloor\frac{n}{2} \right\rfloor}$.
	\end{theorem}
	\begin{proof}
		\begin{figure}[ht!]
			\centering     \includegraphics[width=0.85\linewidth]{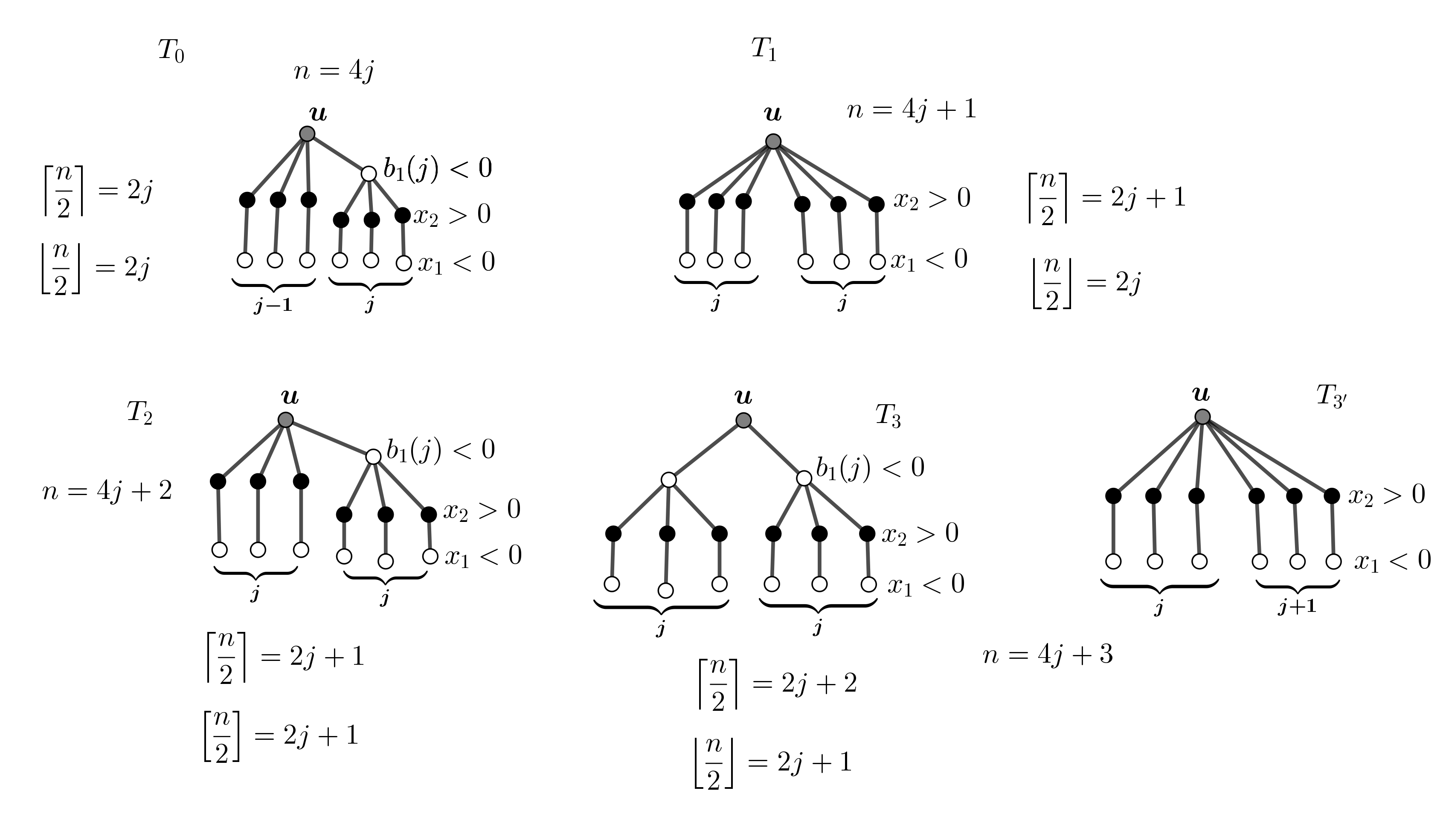}        \caption{Prototype trees}
			\label{fig:prototypes}
		\end{figure}
		In each case, the sign of each output of $\operatorname{Diagonalize}(A_\alpha(T_\theta),-d_\alpha)$, except at the root $u$, is predetermined by Lemma~\ref{lemma-sequence} (c-4) and the fact that $\lfloor \frac{n}{4} \rfloor < r_0$.
		\begin{enumerate}
			\item  For $T_0$, we know that $b_1(j)<0$. So, excluding $u$, we have $2j$ negative outputs and $2j-1$ positive outputs (according to Figure~\ref{fig:prototypes}). Therefore, we need to show that $a_{T_0}(u)>0$, to obtain $\sigma_\alpha(T_0)=2j$. Defining
			\[g_0(\alpha, j):= 
			a_{T_0}(u)= \alpha\,j-d_\alpha -
			{\frac { \left( 1-\alpha \right) ^{2}}{b_{{1}} \left( j \right) }}-
			{\frac { \left( 1-\alpha \right) ^{2} \left( j-1 \right) }{x_{{2}}}},
			\] 
			on $[0,1/2]\times[2,\infty)$, and we must prove $g_0(\alpha, j)>0$. Note that 
			\[g_0(\alpha, j)= -{\frac { \left( 1-\alpha \right) ^{2}}{b_{{1}} \left( j \right) }} +  \alpha (j-1) +(\alpha -d_\alpha ) - {\frac { \left( 1-\alpha \right) ^{2} \left( j-1 \right) }{x_{{2}}}} \]
			Since $-{\frac { \left( 1-\alpha \right) ^{2}}{b_{{1}} \left( j \right) }}>0$ we only need to show that 
			\[\left[ \alpha - \frac{(1-\alpha)^{2}}{x_{2}} \right]  (j-1) + x_1(\alpha) \geq 0, \]
			or equivalently,
			\[j-1 \geq \frac{-x_1(\alpha)}{\alpha -\frac{(1-\alpha)^{2}}{x_{2}}} \geq \frac{-x_1(\alpha)}{\alpha} = \frac{-\alpha (-1 + \frac{2}{4j})}{\alpha} = 1 - \frac{1}{2j} \geq 0,\]
			that is, $j \geq 1$, which always holds.
			
			\item For $T_1$, we have (see Figure~\ref{fig:prototypes}), exactly $2j$ positive and $2j$ negative outputs.
			It remains to show $a_{T_1}(u)<0$. We define
			\[g_1(\alpha, j):= a_{T_1}(u)=2\,\alpha\,j-d_\alpha -2\,{\frac {j \left( 1-\alpha \right) ^{2}}{x_{{2}} }}\] 
			on the rectangle $[0,1/2]\times[2,\infty)$.
			
			Since \(n=4j+1\), we have $d_\alpha
			=\alpha\left(2-\frac{2}{4j+1}\right)
			=\frac{8\alpha j}{4j+1}.$
			Therefore,
			\[
			\begin{aligned}
				g_1(\alpha,j)<0
				&\iff
				2\alpha j
				-2\frac{j(1-\alpha)^2}{x_2}
				<
				\frac{8\alpha j}{4j+1}\\
				&\iff
				2j\left(
				\alpha-\frac{(1-\alpha)^2}{x_2}
				\right)
				<
				\frac{8\alpha j}{4j+1}.
			\end{aligned}
			\]
			Since \(x_2>0\), it is enough to prove that $\frac{(1-\alpha)^2}{x_2} > \frac{\alpha(4j-3)}{4j+1}.$   Using $x_2=\frac{2\alpha}{4j+1} +\frac{(4j+1)(1-\alpha)^2}{\alpha(4j-1)},$ we obtain $(1-\alpha)^2>\frac{2\alpha^2(4j-3)}{(4j+1)^2}+\frac{(4j-3)(1-\alpha)^2}{4j-1}.$
			Thus, $\frac{(1-\alpha)^2}{4j-1}>\frac{\alpha^2(4j-3)}{(4j+1)^2}.$ For \(0<\alpha\leq\frac12\), we have $(1-\alpha)^2\geq\alpha^2$ and $\frac{1}{4j-1} > \frac{4j-3}{(4j+1)^2}.$ 
			Consequently, $\frac{(1-\alpha)^2}{4j-1} > \frac{\alpha^2(4j-3)}{(4j+1)^2},$  which proves that $g_1(\alpha,j)<0.$ 
			
			\item For $T_2$, we define
			\[g_2(\alpha, j):= a_{T_2}(u)=\alpha\, \left( 2\,j+1 \right) -d_\alpha -{\frac { \left( 1-\alpha \right) ^{2}}{b_{{1}}
					\left( j \right) }}-{\frac {j \left( 1-\alpha \right) ^{2}}{x_{{2}}}}
			\]   on the rectangle $[0,1/2]\times[2,\infty)$. We must prove $g_2(\alpha, j)>0$. 
			
			Note that 
			\[g_2(\alpha, j)= \left[\alpha j -{\frac { \left( 1-\alpha \right) ^{2}}{b_{{1}}
					\left( j \right) }} \right] +(\alpha -d_\alpha)+ \alpha j -{\frac {j \left( 1-\alpha \right) ^{2}}{x_{{2}}}}\]
			Since $\alpha j  -{\frac { \left( 1-\alpha \right) ^{2}}{b_{{1}} \left( j \right) }}>0$ we only need to show that $x_1(\alpha) + \alpha j - {\frac { \left( 1-\alpha \right) ^{2} j }{x_{{2}}}} \geq 0$, or equivalently
			\[j \geq \frac{-x_1(\alpha)}{\alpha -\frac{(1-\alpha)^{2}}{x_{2}}} \geq \frac{-x_1(\alpha)}{\alpha} = \frac{-\alpha (-1 + \frac{2}{4j})}{\alpha} = 1 - \frac{1}{2j} \geq 0\]
			that is $j \geq 1$, which always holds.
			
			\item  For $T_3$ and $T_{3'}$, due to the balance of positive and negative outputs, we only need to analyze the output at $u$. Defining the function   
			\[g_3(\alpha, j):=a_{T_3}(u)= 2\,\alpha-d_\alpha -2
			\,{\frac { \left( 1-\alpha \right) ^{2}}{b_{{1}} \left( j \right) }}
			\]  in the interval $[0,1/2]\times [1, \infty)$. We must show that $g_3(\alpha,j)>0$.
			Since $n=4j+3$, we have $d_\alpha
			=
			\alpha\left(2-\frac{2}{4j+3}\right).$
			Consequently, $2\alpha-d_\alpha
			=
			\frac{2\alpha}{4j+3}.$
			Since $b_1(j)<0$, we have $-2\frac{(1-\alpha)^2}{b_1(j)}>0.$ 
			Moreover, $\frac{2\alpha}{4j+3}>0$
			for $\alpha>0$ and $j\geq 1$. Therefore,
			$g_3(\alpha,j)>0$. 
			
			The same holds for \[g_{3'}(\alpha, j):=a_{T_{3'}}(u)= \alpha\, \left( 2\,j+1 \right) -d_\alpha -{\frac { \left( 2\,j+1 \right)  \left( 1-
					\alpha \right) ^{2}}{x_{{2}}}}
			\]
			
			For $n=4j+3$,  $d_\alpha = \alpha\left(2-\frac{2}{4j+3}\right),$
			and we obtain
			\[
			\begin{aligned}
				\alpha(2j+1)-d_\alpha
				&=
				\alpha(2j+1)
				-\alpha\left(2-\frac{2}{4j+3}\right)=
				\frac{\alpha(8j^2+2j-1)}{4j+3}.
			\end{aligned}
			\]
			Thus,
			\[
			g_{3'}(\alpha,j)
			=
			\frac{\alpha(8j^2+2j-1)}{4j+3}
			-\frac{(2j+1)(1-\alpha)^2}{x_2}.
			\]
			
			Therefore,
			\[
			\begin{aligned}
				g_{3'}(\alpha,j)
				&=
				\frac{\alpha(8j^2+2j-1)}{4j+3}-
				\frac{
					\alpha(2j+1)(4j+1)(4j+3)(1-\alpha)^2
				}{
					2\alpha^2(4j+1)
					+(4j+3)^2(1-\alpha)^2
				}.
			\end{aligned}
			\]
			After simplification, we obtain
			\[
			g_{3'}(\alpha,j)
			=
			-\frac{
				2\alpha(2j+1)P(\alpha,j)
			}{
				(4j+3)Q(\alpha,j)
			},
			\]
			where
			\[
			P(\alpha,j)
			=
			24\alpha^2j+10\alpha^2
			-32\alpha j^2
			-48\alpha j
			-18\alpha
			+16j^2+24j+9,
			\]
			and
			\[
			Q(\alpha,j)
			=16\alpha^2j^2+32\alpha^2j+11\alpha^2
			-32\alpha j^2-48\alpha j-18\alpha+16j^2+24j+9.
			\]
			Thus, it remains to prove that $P(\alpha,j)>0$ and $Q(\alpha,j)>0$.
			
			Let $t=1-\alpha$,
			since $0<\alpha\leq\frac12$, we have 
			$ t\geq\frac12.$ $P$ can be written as
			\[
			P(t,j)
			=
			16j^2(2t-1)
			+
			2t(12j+5t-1)+1.
			\]
			Since $t\geq\frac12$, $16j^2(2t-1)\geq0$, and $2t(12j+5t-1)+1>0$. Therefore, $P(\alpha,j)>0$. Moreover, after substituting $\alpha=1-t$,
			\[
			Q(t,j)=16j^2t^2+16jt(2t-1)+8j+11t^2-4t+2>0,
			\]
			for $t\geq\frac12$ and $j\geq1$. Hence $g_{3'}(\alpha,j)<0$. 
		\end{enumerate}
	\end{proof}

	\section{Reduction: starlike vertices} \label{sec:reduction}
	The next lemma is identical to \cite{Jacobs2021} since it does not depend on the matrix we are using:
	\begin{lemma}\cite[Lemma 6.2]{Jacobs2021}\label{lem:reduction-starlike-vertex}
		Let $T$ be a tree with $n \ge 7$ vertices and $k \ge 2$ starlike vertices. Then there exists a starlike vertex $u$ such that $w(u) \le 2\lfloor \frac{n}{4} \rfloor$.
	\end{lemma}
	
	The next theorem is proved exactly in the same way as in \cite{Jacobs2021} using the fact that by Lemma~\ref{lemma-sequence}, the inequality $ \lfloor \frac{n}{4}\rfloor\leq \frac{n}{4}<r_0$ and the proper transformations Star-Down and Star-Star are now adapted to stars $S_r$ with $r \leq r_0$ (see Proposition~\ref{prop:starup} and Proposition~\ref{prop:stardown-transf}).
	\begin{theorem}\cite[Theorem 6.1]{Jacobs2021} \label{annihilation}
		Consider a tree $T$ in $(P_q, S_r)$ representation and $u$ a starlike vertex
		with $\ell \geq 2$ generalized pendant paths, or $P(u)= P_{q_1}\ast S_{r_1}
		\oplus \cdots \oplus P_{q_\ell }\ast S_{r_\ell}$. If $w(u)\leq2\lfloor
		r_0\rfloor $ then we can properly transform $T$ to $T^\prime$ obtaining
		$$P_{q_1}\ast S_{r_1} \oplus \cdots \oplus P_{q_\ell}\ast S_{r_\ell} \Rightarrow  P_{q' }\ast S_{r'}$$
		where
		\[  \left\{
		\begin{array}{l}
			q'\equiv \sum q_i \pmod 2 \in \{0,1\} \\
			r' =\frac{w(u) -q'}{2} \leq \lfloor r_0\rfloor
		\end{array}
		\right.
		\]
	\end{theorem}
	
	The reasoning in the next theorem is similar, but not identical to \cite{Jacobs2021}. However, the introduction of the new prototype $T'_3$ keeps only the combinatorial part requiring no adaptation. More precisely, in \cite[Theorem 7.2]{Jacobs2021} one observes that if there are no starlike vertices, then, using only the proper transformations, one can reduce $T$ to $T_0, T_1, T_2$ or $T_3$, except for one case when $n \equiv 3 \mod 4$ where one obtains a tree $T'_3$, and an additional proof is required to properly transform it in $T_3$. In our case, we avoid this by including this tree as a prototype and checking that it satisfies the conjecture. The case of one starlike vertex was proved in \cite[Theorem 7.3]{Jacobs2021} using exclusively the proper transformations and combinatorial reasoning. Therefore, no adaptation is necessary when moving from the Laplacian to the $A_\alpha$ matrix. Summarizing these considerations, we state the following key result:
	\begin{theorem} \label{thm:zero-one} Let $T$ be a tree of order $n$ having no starlike vertices or exactly one starlike vertex. Then $T$ can be properly transformed into $T_\theta, ~\theta \in \{0,1,2,3\}$ according to $n\equiv\theta\pmod4$.
	\end{theorem}

	We are now able to prove our main result.
	\begin{proof}(Of  Theorem~\ref{thr:main}) 
		For $n\leq 6$ the proof follows by examining every tree individually. This is done in Corollary~\ref{cor:conjecture_small_n}.
		
		For $n\geq7$, we assume that a tree has $n \geq 7$ vertices and consider $A_\alpha$ for $0<\alpha\leq \frac{1}{2}$ (we already checked for $\alpha=0$). Then, we start the algorithm \texttt{Transform}$(T)$ (see Figure~\ref{fig:transform}). In turn, it calls the algorithm \texttt{InitiateRepresentation}($T$)(see Figure~\ref{fig:init-rep}), orders the starlike vertices of $T$ from the smallest weight $u_1$ to the largest weight $u_k$, and then calls the algorithm \texttt{ReduceStarVertex}($T,u_1$)(see Figure~\ref{red-star}) to convert all gpp's at $u_1$ into a single one.
		
		Provided $k\geq 2$, Lemma~\ref{lem:reduction-starlike-vertex} shows that $w(u_1) \le 2\lfloor \frac{n}{4} \rfloor$. In turn, this allows us to apply Theorem~\ref{annihilation}. This procedure does not necessarily decrease the number of starlike vertices, but eventually increases the weight of the lighter one. Therefore, eventually we must have no starlike vertices or exactly one starlike vertex, otherwise we will get a contradiction with $k\geq 2$ and $w(u_1) > 2\lfloor \frac{n}{4} \rfloor$.
		
		Thus, we must consider both situations and apply Theorem~\ref{thm:zero-one} to properly transform $T$ into one of the prototype trees $T_\theta$. By Theorem~\ref{thr:Talpha}, every prototype tree $T_\theta$ satisfies  ${\displaystyle\sigma_\alpha(T_\theta)
			= \left \lfloor\frac{n}{2} \right\rfloor}$. This concludes our proof.
	\end{proof}
	
	\section{Further generalizations}\label{sec:further}
	In the same way that the $A_\alpha$ matrix combines the Laplacian and adjacency matrices of a graph, some other deformations of these two matrices are well known and frequently studied.

	Samanta, Deepshikha, and Das (see \cite{SamantaDeepshikhaDas}) introduced the one-parameter family
	\begin{equation}\label{eq:B-definition}
		B_{\beta}(G)=\beta A(G)+(1-\beta)L(G),
		\qquad 0\leq\beta\leq1.
	\end{equation}
	Some particular values are
	\[
	B_0=L,\qquad B_{1/2}=\tfrac12D,\qquad
	B_{2/3}=\tfrac13Q,\qquad B_1=A.
	\]
	
	Another important family is the deformed Laplacian introduced by F.~Morbidi in 2013~\cite{Mor2013}, within the context of a generalized continuous-time consensus protocol (see also \cite{GDV2018} who proposed a centrality measure based on nonbacktracking walks):
	\begin{equation}\label{eq:deformed}
		M_G(s)=I-sA(G)+s^2(D(G)-I), s \in \mathbb{R},
	\end{equation}
	has also been studied by \cite{Diaz2026spectral,Oliveira2027deformed} from the limit-point perspective. Some particular values are $M_G(0)=I$, $M_G(-1)=Q$ and $M_G(1)=L$.
	
	One can ask whether Conjecture~\ref{thm:main-conjecture} holds for these families. The next proposition shows a direct application of our results via reparametrization for some values of $\beta$ and $s$. The remaining cases are a subject of future investigation. We would like to thank Prof. Francesco Belardo, who informed us of these relations, which will appear as part of ongoing research on the distribution of Laplacian limit points.
	\begin{proposition} \label{prop:other-families} The following statements hold:
		\begin{enumerate}
			\item For a tree $T$, let $t=|s|>0$, and define
			\[
			\beta_t:=\frac{1+t}{1+2t},\qquad
			\alpha_t:=\frac{t}{1+t}.
			\] 
			Then,
			\begin{equation}\label{eq:deformed-B-A-equivalence}
				M_T(-t)= (1-t^2)I+t(1+2t)B_{\beta_t}(T)
				=(1-t^2)I+t(1+t)A_{\alpha_t}(T).
			\end{equation}
			\item The analogue of Conjecture~\ref{thm:main-conjecture} holds for $B_{\beta}(T)$, for  $ \frac{2}{3} \leq\beta\leq 1$;
			\item The Conjecture~\ref{thm:main-conjecture} is true for $M_T(s) $, for  $ -1 \leq s \leq 1$.
		\end{enumerate}
	\end{proposition}
	\begin{proof}
		\begin{enumerate}
			\item  The proof of Equation~\ref{eq:deformed-B-A-equivalence} follows directly by checking the formulas against the definition of each type of matrix (see Equations~\ref{eq:B-definition} and \ref{eq:deformed}).  Since $T$ is a tree, it is bipartite. Hence, $M_T(t)$ and $M_T(-t)$ have the same spectrum when $t=|s|$ (see \cite{Diaz2026spectral} for details).
			\item  The proof is obtained by establishing a correspondence between $B_{\beta}(T)$ and the matrix $A_\alpha$, then applying Theorem~\ref{thr:main}.
			
			Analogously, solving $\alpha_t:=\frac{t}{1+t}$ for $t$ in Equation~\ref{eq:deformed-B-A-equivalence}, one can see that, if $\alpha<1$, then 
			\[B_{\beta_\alpha}(T) =\beta_\alpha A_{\alpha}(T),\]
			where $\beta_\alpha= \frac{1}{1+\alpha}$. Thus, $0\leq \alpha \leq \frac{1}{2}$ if and only if $\frac{2}{3}\leq \beta_\alpha \leq 1$.
			In particular,
			\[d_{\beta_\alpha}=\frac{1}{n} \operatorname{tr}(B_{\beta_\alpha}(T)) =\beta_\alpha \frac{1}{n} \operatorname{tr}(A_{\alpha}(T))= \beta_\alpha d_\alpha.\]
			Moreover, $A_\alpha v = \lambda v$ if and only if  $\beta_\alpha A_\alpha v = \beta_\alpha \lambda v$ if and only if $B_{\beta_\alpha}(T) v = \beta_\alpha \lambda v$, therefore, $\lambda$ is an eigenvalue of $A_\alpha(T)$ if and only if $\beta_\alpha \lambda$  is an eigenvalue of $B_{\beta_\alpha}(T)$. Together, these observations show that if there are at least $\lceil \frac{n}{2}\rceil$ $A_\alpha$-eigenvalues less than or equal to $d_\alpha$ then we also have at least $\lceil \frac{n}{2}\rceil$ $B_{\beta_\alpha}$-eigenvalues less than or equal to $d_{\beta_\alpha}$. Consequently, Conjecture~\ref{thm:main-conjecture} is true for $B_{\beta}$ provided that $\frac{2}{3}\leq \beta  \leq 1$.

			\item   If $t=0$ then $s=0$ and $M_T(0)=I$. Thus, $d_0=1$, $\operatorname{Spec}(M_T(0))=\{1^n\}$ satisfying the Conjecture~\ref{thm:main-conjecture} trivially.

			Otherwise, let $0<t\leq1$. Since $T$ is a tree, it is bipartite. Hence, for the diagonal signature matrix $J$ associated with a bipartition, $M_T(s)$ is similar to $M_T(-s)$. Thus it suffices to consider $M_T(-t)$, and by Equation~\ref{eq:deformed-B-A-equivalence},
			
			\[M_T(-t)=(1-t^2)I+t(1+t)A_{\alpha_t}(T),\qquad \alpha_t:=\frac{t}{1+t}.\]
			
			Note that $0\leq\alpha_t\leq\frac12$ if and only if $t\leq1$. Define $d_s$ as the average of the eigenvalues of $M_T(s)$. By similarity, $M_T(s)$ and $M_T(-t)$ have the same trace, so
			\[d_s=\frac{1}{n}\operatorname{tr}(M_T(s))
			=(1-t^2)+t(1+t)d_{\alpha_t}.\]
			
			If $A_{\alpha_t}(T)v=\lambda v$, then
			\[M_T(-t)v=\bigl[(1-t^2)+t(1+t)\lambda\bigr]v.\]
			Therefore, the eigenvalues of $M_T(-t)$ are obtained from those of $A_{\alpha_t}(T)$ by the affine transformation
			\[\lambda'=(1-t^2)+t(1+t)\lambda.\]
			Since $t(1+t)>0$,
			\[\lambda'\leq d_s
			\iff \lambda\leq d_{\alpha_t}.\]
			By the main theorem, at least $\lceil n/2\rceil$ eigenvalues of $A_{\alpha_t}(T)$ are at most $d_{\alpha_t}$. Hence the same is true for $M_T(-t)$ and, by similarity, for $M_T(s)$. This proves the claim.
		\end{enumerate}
	\end{proof}

	\begin{remark}
		A direct consequence of Proposition~\ref{prop:other-families} is that all the counterexamples we built in Proposition~\ref{prop:counterexamples} are immediately transferred to the matrices $M_T(s)$ and $B_{\beta}(T)$.
	\end{remark}  
	
	\textbf{Use of artificial intelligence:}  Artificial intelligence was used solely for language editing and proofreading. The authors have reviewed the manuscript thoroughly and accept full responsibility for its content and any remaining errors.
	
	\textit{\textbf{Acknowledgements:} 
		Elvia P\'erez thanks CAPES for the doctoral scholarship
		that made this work possible. This research is part of her doctoral thesis. Elismar R. Oliveira is partially supported by  CNPq grant 408180/2023-4.}

\end{document}